\documentclass[12pt,centertags,reqno]{amsart}
\usepackage{amssymb}
\usepackage{amsmath}%
\usepackage{amsthm}
\usepackage{braket}
\usepackage{mathtools}
\usepackage{mathrsfs}
\usepackage{bbm}
\usepackage{xcolor}
\usepackage{amsfonts}%
\usepackage{color,graphicx}
\usepackage[utf8]{inputenc}
\usepackage[T1]{fontenc}
\usepackage[english]{babel}
\usepackage{hyperref}
\usepackage{latexsym}
\usepackage{fancyhdr}
 \usepackage[foot]{amsaddr}
 \usepackage[numbers]{natbib}
 \usepackage{comment}

\usepackage{orcidlink}
\DeclarePairedDelimiter{\abs}{\lvert}{\rvert}
\DeclarePairedDelimiter{\norma}{\lVert}{\rVert}
\newcommand{\spL}{H}

\mathtoolsset{showonlyrefs}
 \newtheorem{thm}{Theorem}[section]
 
 \newtheorem{lem}[thm]{Lemma}
 \newtheorem{prop}[thm]{Proposition}
 \theoremstyle{definition}
 \newtheorem{defn}[thm]{Definition}
 \theoremstyle{remark}
 \newtheorem{rem}[thm]{Remark}
 \newtheorem*{ex}{Example}
 \numberwithin{equation}{section}

\begin{document}

%
%
%
%
%
%
%
%
%
\author[L.~Angelini]{Lucia Angelini}
\address{Lucia Angelini, Department of Mathematics and Computer Science, University of Perugia, Via Luigi Vanvitelli, 1, I-06123 Perugia, Italy.}\email{lucia.angelini92@gmail.com}
\author[I.~Benedetti]{Irene Benedetti\, 
}
\address{Irene Benedetti, Department of Mathematics and Computer Science, University of Perugia, Via Luigi Vanvitelli, 1, I-06123 Perugia, Italy.}\email{irene.benedetti@unipg.it}
\author[A.~Cretarola]{Alessandra Cretarola\, 
}
\address{Alessandra Cretarola, Department of Economic Studies,
 University "G. D'Annunzio'' of Chieti-Pescara,
Viale Pindaro, 42, I-65127 Pescara, Italy.}\email{alessandra.cretarola@unich.it}
\title[Existence of solutions of stochastic evolution inclusions]
 {Existence of nonnegative mild solutions of stochastic evolution inclusions via weak topology}


\date{}

\begin{abstract}
This paper addresses the existence of nonnegative mild solutions for stochastic evolution inclusions through a weak topology approach. Precisely, the study focuses on stochastic evolution inclusions characterized by multivalued nonlinearities and perturbed by a $Q$-Wiener process within a Hilbert space framework. The primary objective is to establish the existence of mild solutions under conditions that ensure sublinear growth for the involved functions and multivalued mappings. By means of the weak topology method, we establish the existence of at least one mild solution in an appropriate function space. Additionally, when the associated semigroup is positive and a specific sign condition is met, we show the existence of nonnegative solutions. The methodological approach involves approximating the problem via truncation on bounded intervals and applying fixed-point theorems in weak topology to extend local solutions to the entire half-line. 
\end{abstract}

\maketitle


\section{Introduction}
Stochastic inclusions naturally arise as reduced or theoretical formulations of stochastic control problems, see, for example, the seminal work \cite{aubin1998viability}. Following the approach outlined in that paper, let $\mathbb{R}^n$ denote the state space of a control system and $\mathbb{R}^m$ the space in which control values lie. The dynamics of the system is governed by two mappings: $f : [0,+\infty) \times \mathbb{R}^n \to \mathbb{R}^n$, $g : [0,+\infty) \times \mathbb{R}^n \times \mathbb{R}^m \to \mathcal{L}(\mathbb{R}^m, \mathbb{R}^n)$, where $\mathcal{L}(\mathbb{R}^m, \mathbb{R}^n)$ denotes the space of linear maps from $\mathbb{R}^m$ to $\mathbb{R}^n$. The feedback condition, which restricts the admissible choices of control at a given time and state of the system, is expressed by the relation $v_t(\omega) \in V(u_t(\omega))$, with $t > 0$ and $\omega \in \Omega$, where $(\Omega,\mathcal{A},\mathbf P)$ is a probability measure space and $V$ is a prescribed set. The resulting description of the control system can thus be formulated as:
\begin{equation}
\begin{array}{l}
 du_t = f(t, u_t)dt + g(t,u_t,v_t) dW_t, \quad t > 0, \\
 \mbox{for almost all} \; (\omega,t) \in \Omega \times [0,+\infty), v_t(\omega) \in V(u_t(\omega)),
 \end{array}
\end{equation}
where $W=(W_t)_{t \geq 0}$ is a given $\mathbb{R}^m$-valued Wiener process. The problem described above fits into the following framework:
\begin{equation}
   du_t = f(t, u_t)dt + G(t,u_t) dW_t, \quad t > 0, 
\end{equation}
where $G: [0,+\infty) \times \mathbb{R}^n \to P(\mathcal L(\mathbb{R}^m,\mathbb{R}^n))$ is a set-valued map defined by $G(t,u_t)=g(t,u_t,V(u_t))$, here, given a space $E$, $P(E)$  refers to the collection of all nonempty subsets of the space $E$. 

\noindent The authors in \cite{aubin1998viability} study differential inclusions in a finite-dimensional setting, where the set-valued drift and diffusion terms are jointly upper semicontinuous. An alternative set of assumptions on the set-valued integrands involves lower semicontinuity with respect to the second variable; see \cite{kisielewicz2001weak}, where the author proves the existence of solutions for stochastic differential inclusions, still within a finite-dimensional framework. We refer to the monograph \cite{kisielewicz2013stochastic} for a comprehensive study of stochastic differential inclusions in the finite dimensional setting.

\noindent We consider a semilinear stochastic differential inclusion in infinite dimensional Hilbert spaces. Namely, the problem under consideration is the following
\begin{equation}
\label{inclusione}
\begin{cases}
du_t \in [Au_t + f(t, u_t)]dt + \Sigma(t,u_t)dW_t, \quad t > 0 \\
u(0)=u_0,
\end{cases}
\end{equation}
where $W=(W_t)_{t \geq 0}$ is a given $K$-valued $Q$-Wiener process, with $Q:K \to K$ a trace class operator, where $K$ is a separable Hilbert space, $u_0 \in L^2(\Omega,H)$, $A:D(A) \subset H \to H$ is the infinitesimal generator of a strongly continuous semigroup $\{T(t)\}_{t\ge 0}$ in $H$, $f : [0,+\infty) \times \spL \to \spL$ is a linear operator, and $\Sigma : [0,+\infty) \times \spL \to P(\mathcal L_2^0)$ is a multivalued nonlinear operator, with $H$ another separable Hilbert space and $\mathcal L_2^0$ denoting the space of Hilbert-Schmidt operators.

\noindent For the problem \eqref{inclusione} we prove the existence of at least one mild solution on the whole half line. Additionally, when the associated semigroup is positive and a specific sign condition is met, we show the existence of nonnegative solutions. 

\noindent The existence of solutions to these problems is often investigated using topological methods, particularly fixed point theorems applied to an appropriate solution operator. However, such approaches typically require strong compactness assumptions, which are generally not satisfied in infinite-dimensional settings when the evolution operator associated with $A$ is not compact. As an alternative, Lipschitz-type conditions can be imposed on the drift and diffusion terms. The condition that the linear part generates a compact semigroup is adopted in \cite{toufik2015existence}, in the context of standard stochastic differential inclusions; in \cite{balasubramaniam2006controllability}, which deals with nonlinear neutral stochastic differential inclusions; and more recently in \cite{dineshkumar2022discussion} and \cite{kumar2024new}, in the setting of fractional stochastic differential inclusions. In all these works, the diffusion term is multivalued. On the other hand, in \cite{da1994stochastic}, the existence of both strong and mild solutions to an inclusion of the form \eqref{inclusione} is established under a global Lipschitz condition. This assumption is relaxed in \cite{jakubowski2005existence}. In \cite{ren2011controllability}, a global Lipschitz condition is again employed to establish both existence and controllability of solutions in the setting of neutral stochastic differential inclusions.

\noindent Here, we apply a technique based on weak topology, that allows us to avoid hypotheses of compactness both on the semigroup generated by the linear part and on the terms $f$ and $\Sigma$ and does not require any Lipschitz continuity assumptions on the nonlinearities. In particular, this approach allows us to treat a class of nonlinear multivalued maps $\Sigma$ which are not necessarily compact-valued. This technique was developed for deterministic semilinear differential inclusions in \cite{benedetti2013nonlocal} and was used for the same type of stochastic differential inclusions of \eqref{inclusione} in \cite{zpa2018}, see also \cite{zhou}. Specifically, in \cite{zpa2018}, the authors prove the existence of at least one mild solution for the problem \eqref{inclusione} restricted on a bounded interval $[0,b]$. They consider the same regularity assumptions as we do, see hypotheses $(H_A),(H_1)-(H_2)-(H_4)-(H_5)$, however they assume stronger growth conditions on the diffusion term $\Sigma$. In particular, in \cite[Theorem 4.1]{zpa2018} the authors assume the following assumption:
\begin{itemize}
\item[($H_6^\prime$)] for every $n>0$, there exists a function $\mu_n \in L^1([0,b], \mathbb{R}^{+})$ such that for each $u \in \spL$, $\norma {u}_H^2 \le n$, 
        \begin{equation*}
        \norma {\Sigma(t, u)}^2_{\mathcal{L}_2^0} \le \mu_n(t) \quad \text{for a.e.} \quad t \in [0,b],
        \end{equation*}
        and
\begin{equation}
    \label{eq main}
\liminf_{n \to \infty} \frac{1}{n} \int_0^{+\infty} \mu_n(s) \,ds = 0.
\end{equation}
\end{itemize}
Notice that conditions $(H_6^\prime)$ imply the sublinear growth conditions $(H_6)$ that we require:
\begin{itemize}
     \item[($H_6$)] there exists a function $C_\Sigma \in L^1([0,+\infty),\mathbb{R}_+)$ such that $\norma{\Sigma(t,u)}_{\mathcal{L}_2^0}^2  \le C_\Sigma(t) (1 + \norma{u}_\spL) $ for all $t \in [0,+\infty)$ and $u \in \spL$.
\end{itemize}
To show this implication, we notice that by \eqref{eq main}, we can consider a subsequence, denoted as the sequence $\{\mu_n\}$ such that the condition \eqref{eq main} yields again. Hence, we have that the subsequence $\{\frac{1}{n} \mu_n \}$ converges to the null function in $L^1([0,+\infty), H)$. Then, there exist a subsequence, called as the sequence $\{\frac{1}{n} \mu_n \}$ and a function $h \in L^1([0,+\infty), \mathbb{R}^+)$ such that 
\begin{equation}
\label{cond. ipotesi}
    \frac{1}{n}\mu_n(t) \le h(t) \quad \text{for all} \, \, n \, \, \text{and a.e.} \,\, t \in [0,+\infty),
\end{equation}
see \cite[Theorem 4.9]{brezis}. We fix $t \in [0,+\infty)$ such that verifies the condition \eqref{cond. ipotesi}. Let $u \in H$ and $n = n(u) = 1, 2, \dots$ such that $n -1 \le \norma{u}_H \le n$. By \eqref{cond. ipotesi} we have
\begin{equation*}
     \norma {\Sigma(t, u)}^2_{\mathcal{L}_2^0} \le \mu_n(t) \le nh(t) \le h(t) (1+ \norma{u}_H), \,  \text{a.e.} \,\, t \in [0,+\infty),
\end{equation*}
 obtaining the claimed result. In all the other existence results proven in \cite{zpa2018} a sublinear condition such as $(H_6)$ is required as well, but is always associated to some extra conditions that are not needed in our proofs. Furthermore, when comparing our results with those in \cite{zpa2018}, we found it necessary to assume the linearity of the term \( f \) and the convexity of the set-valued mapping \( \Sigma \) with respect to the second variable. This requirement arises from the need to properly handle regularity with respect to the weak topology. In particular, to prove the closedness of the solution operator, one must consider a sequence \( u^m \) converging weakly in \( L^2(\Omega, H) \). However, both \( f \) and \( \Sigma \) are regular only with respect to the weak topology in \( H \), and it is not true in general that a sequence converging weakly in \( L^2(\Omega, H) \) admits a subsequence converging weakly in \( H \) almost surely. For instance, the sequence 
\[
u^m(\omega) = \sin(m\omega)
\]
converges weakly in \( L^2([-\pi,\pi], \mathbb{R}) \), but has no subsequence converging almost everywhere in \( \mathbb{R} \). To overcome this issue, we introduced the assumptions of linearity of \( f \) and convexity of \( \Sigma \) in the second variable.

\noindent The existence of a nonnegative mild solution in the large is also established in \cite{marinelli} and \cite{marinelli2022positivity}. However, both works consider single-valued drift and diffusion terms. In \cite{marinelli2022positivity}, the semigroup generated by the linear part is assumed to be positive and sub-Markovian, while in \cite{marinelli}, the semigroup is only assumed to be positive, but the drift and diffusion terms are required to be Lipschitz continuous.

\noindent In \cite{marinelli}, the assumption that $A$ generates a positive semigroup can be made without additional conditions, since the Lipschitz continuity of the nonlinear terms ensures the existence of at least one mild solution to the corresponding stochastic differential equation, as well as the convergence of the sequence of strong solutions of the regularized problems obtained by replacing $A$ with its Yosida approximation.

\noindent In the more general setting considered in \cite{marinelli2022positivity}, where drift and diffusion coefficients are not assumed to be Lipschitz continuous, it is not even clear whether the regularized equation admits a solution at all. For this reason, the authors of \cite{marinelli2022positivity} introduce a different approximation scheme and require that $A$ generates a sub-Markovian semigroup, rather than merely a positive one.

\noindent We highlight the fact that our approach allows us to dispense with both the sub-Markovian property of the semigroup generated by $A$ and the Lipschitz continuity assumptions on the drift and diffusion coefficients. Moreover, in contrast to \cite{marinelli} and \cite{marinelli2022positivity}, we are able to handle a multivalued diffusion term, thus significantly broadening the class of admissible stochastic perturbations. 

\noindent As a final remark, we emphasize that, under the assumptions considered in this paper, we are able to prove that for every $k \in \mathbb{N}$ and every $u \in \mathcal{C}^k$, the set of selections of the multivalued map $\Sigma$
\begin{equation*}
Sel_{\Sigma}^k(u) = \left\{\sigma \in L^2_{\mathscr{F}}([0,k], \mathcal L_2^0) : \sigma_t \in \Sigma(t, u_t) \quad \text{for a.e.} \, t \in [0,k] \right\}
\end{equation*}
is nonempty, where $L^2_{\mathscr{F}}([0,k], \mathcal L_2^0)$ is the space of all $\mathbb{F}$-progressively measurable processes defined on $[0,+\infty)$ with values in $\mathcal{L}_2^0$. This result is deduced from our hypotheses and is not assumed a priori. In contrast, many works in the literature take the nonemptiness of $Sel_{\Sigma}^k(u)$ for granted for every $u \in \mathcal{C}^k$; see, for instance, \cite{dineshkumar2022discussion, dineshkumar2022results, kumar2024new, toufik2015existence}.

\noindent The remainder of the paper is structured as follows. Section \ref{sec:prelim} gathers background material on the theory of semigroups, trace‑class and Hilbert–Schmidt operators, multivalued mappings, and stochastic processes that will be used throughout. In Section \ref{sec:SEI} we state the stochastic evolution inclusion that is the focus of the study and establish our first main result: existence of a mild solution on the infinite horizon.
Section \ref{sec:approx} introduces a family of truncated problems whose analysis paves the way for extending the existence result to the entire non‑negative half‑line, a task achieved in Section \ref{sec:proof}.
Finally, Section \ref{sec:nonneg} contains our second main result, proving that the mild solution is non‑negative under suitable conditions.

\section{Some preliminaries}\label{sec:prelim}

\noindent Let $X$ be a Banach space with the norm $\norma{\, \cdot \,}_X$ and $J = [a,b] $, with $- \infty < a < b < \infty$, be a finite interval of $\mathbb{R}$. Denote by $C(J,X)$ the Banach space of all continuous functions from $J$ into $X$ with the norm 
\begin{equation}
    \norma{u}= \sup_{t \in J} \norma{u(t)}_X,
\end{equation}
where $u \in  C(J,X)$. Let $1\le p \le \infty$. We denote by $L^{p}(J,X)$ the set of all strongly measurable functions $f: J \to X$ such that $\norma{f}_{L^p(J,X)} < \infty$, where
\begin{equation}
    \norma{f}_{L^p(J,X)} = 
    \begin{cases}
    (\int_J \norma{f(t)}_X^p \, dt)^{\frac{1}{p}} & 1 \le p < \infty \\
    {\rm ess} \sup_{t \in J}\norma{f(t)}_X & p=\infty.
    \end{cases}
\end{equation}
We denote with $\|\cdot\|_p$ the usual norm in the space $L^p(\Omega,\mathbb{R})$, with $1 \leq p \leq \infty$ and $\Omega \subset \mathbb{R}^n$ a $\sigma$-finite measurable set.

\noindent We introduce the weak topology and recall some well known results; for the proofs and further details we refer to \cite{brezis}.

\begin{thm}
\label{zhou th 1.26}
Let $D$ be a subset of a Banach space $X$. The following statements are equivalent:
\begin{itemize}
\item[(i)] $D$ is relatively weakly compact;
\item[(ii)] $D$ is relatively weakly sequentially compact.
\end{itemize} 
\end{thm}


\begin{thm}
\label{zhou th 1.27}
The convex hull of a weakly compact set in a Banach space $X$ is weakly compact.
\end{thm}

\noindent Finally, in the sequel we will use the following result that is an easy consequence of \cite[Corollary 2.6]{diestel1993weak}
\begin{thm}
    \label{weakcompL1}
Let $(E,\mathcal{E})$ be a finite measure space and $K \subset L^1(E,X)$ be a uniformly integrable set of functions with $X$ a reflexive Banach space. Then, $K$ is relatively weakly compact in $L^1(E,X)$. 
\end{thm}

\subsection{Theory of Semigroups}

\noindent In this section we introduce the class of semigroups of linear operators in Hilbert spaces, for more details we refer to \cite{vrabie2003} and \cite{engel2006short}.

\noindent Let $H$ be a Hilbert space and $\mathcal{L}(H)$ be the Banach space of linear bounded operators with the norm $\norma{\, \cdot \,}_{\mathcal{L}(H)}$.

\begin{defn}
A one parameter family $\{T(t)\}_{t \ge 0} \subset \mathcal{L}(X)$ is a \emph{semigroup} of bounded linear operators on $H$ if
\begin{itemize}
    \item[(i)] $T(t)T(s) = T(t+s)$, for $t,s \ge 0$; 
    \item[(ii)] $T(0)=I$; here $I$ denotes the identity operator in $X$.
\end{itemize}
\end{defn}

\begin{defn}
A semigroup of linear operator $\{T(t)\}_{t \ge 0}$ is called \emph{strongly continuous}, or a \emph{semigroup of class} $C_0$, or $C_0$-\emph{semigroup}, if the mapping $t \mapsto T(t)x$ is strongly continuous, for each $x \in H$ i.e.
\begin{equation*}
    \lim_{t \to 0} T(t)x = x \quad  \forall \, x \in H.
\end{equation*}
\end{defn}

\begin{defn}
Let $\{T(t)\}_{t \ge 0}$ be a $C_0$-semigroup defined on $H$. The linear operator $A$ is the \emph{infinitesimal generator of} $\{T(t)\}_{t \ge 0}$ defined by
\begin{equation*}
    Ax= \lim_{t \to 0} \frac{T(t)x -x}{t} \quad \text{for} \quad x \in D(A),
\end{equation*}
where $D(A)= \Big\{ x \in X : \lim_{t \to 0} \frac{T(t)x -x}{t} \,\, \text{exists in} \,\, H \Big\}$.    
\end{defn}

\begin{thm}
If $\{T(t)\}_{t \ge 0}$ is a $C_0$-semigroup, then there exists $M \ge 1$, and $w \in \mathbb{R}$ such that 
\begin{equation}
\label{eq vrabie 2.3.1}
    \norma{T(t)}_{\mathcal{L}(H)} \le M e^{tw}
\end{equation}
foe each $t \ge 0$. 
\end{thm}
\noindent If $M=1$ and $w=0$, the semigroup $\{T(t)\}_{t \ge 0}$ is said semigroup of contractions.

\begin{defn}
\label{positive semigroup}
A $C_0$-semigroup on a Banach lattice $X$ is called positive if each operator $T(t)$ is positive, i.e.
\begin{equation}
x \geq 0 \quad \mbox{implies} \quad T(t)x\geq 0 \quad \mbox{for every} \quad x \in X \quad{and} \quad t \geq 0.
\end{equation}
\end{defn}
\noindent As usual, given a linear and closed operator $A : D(A) \subseteq H \to H$ we denote by $\rho(A)$ the set of \emph{regular values of the operator} $A$, i.e. the set of all $\lambda \in \mathbb{C}$ for which the range of $\lambda I - A$ is dense in $H$ and $( \lambda I - A)^{-1} : R( \lambda I - A) \to D(A)$ is continuous. The mapping $R(\, \cdot \,; A) : \rho(A) \to \mathcal{L}(H)$ defined by $J^\lambda = (\lambda I - A)^{-1}$ for each $\lambda \in \rho(A)$, is called the \emph{resolvent function of} $A$.
In the next proposition we collect some classical properties of the generator of a semigroup of contractions.
\begin{prop}
\label{propertiesJlambda}
Let $A : D(A) \subseteq X \to X$ be the infinitesimal generator of a $C_0$-semigroup $\{T(t)\}_{t \ge 0}$ of contractions, then 
\begin{itemize}
    \item[(i)] $D(A)$ is dense in $X$;
    \item[(ii)] $A$ is a closed operator;
    \item[(iii)] $(0, + \infty) \subseteq \rho(A)$;
    \item[(iv)] for each $\lambda > 0$ $\norma{J_\lambda}_{\mathcal{L}(X)} \le \displaystyle\frac{1}{\lambda}$;
    \item[(v)] $\displaystyle\lim_{\lambda \to \infty}$ $\lambda J_\lambda x = x$ for every $x \in H$;
    \item[(vi)] $A$ is a dissipative operator, i.e. $\langle Ax,x \rangle \leq 0$ for every $x \in H$.
\end{itemize}
\end{prop}

\subsection{Trace Class and Hilbert-Schmidt Operators}
\label{trace class and H-S space}
Throughout this section, we follow \cite{fabbrigozzi}; $U$, $V$, $Z$ will denote real, separable Hilbert spaces, $\braket{\cdot, \cdot}_U$, $\braket{\cdot, \cdot}_V$, $\braket{\cdot, \cdot}_Z$, and $\|\cdot\|_U, \|\cdot\|_V, \|\cdot\|_Z$, will be, respectively, the inner products and the relative norms in $U$, $V$ and $Z$. We denote by $\mathcal{L}(U,Z)$ the set of all bounded, continuous and linear operators $T : U \to Z$.

\begin{defn}
\label{traceclass}
A linear operator $T \in \mathcal{L}(U,Z)$ is called \emph{trace class} if $T$ can be represented by the form 
\begin{equation*}
    T(z) = \sum_{k=1}^{+ \infty} b_k \braket{z, a_k}_U \quad \text{for any} \, z \in U,
\end{equation*}
where $a_k$ and $b_k$ are two sequences of elements, respectively, in $U$ and $Z$ such that $\sum_{k=1}^{+ \infty} \norma{a_k}_U \norma{b_k}_Z < + \infty$. We denote by $\mathcal{L}_1(U, Z)$ the set of all trace class operators from $U$ to $Z$. 
\end{defn}

\begin{prop}
$\mathcal{L}_1(U,Z)$ is a separable Banach space with respect to the norm
\begin{align*}
    & \norma{T}_{\mathcal{L}_1(U, Z)} \\
    &= \inf \Big\{\sum_{k=1}^{+ \infty} \norma{a_k}_U \norma{b_k}_Z : \{a_k\} \subset U, \{b_k\} \subset Z,\; T(z)= \sum_{k=1}^{+ \infty} b_k \braket{z, a_k}_U, \forall z \in U \Big\}.
\end{align*}
\end{prop}

\begin{prop}
Given $T \in \mathcal{L}_1(Z,Z)$ and an orthonormal basis $\{e_k\}_{k \in \mathbb N}$ of $Z$, the series
\begin{equation*}
    \sum_{k=1}^{+\infty} \braket{Te_k, e_k}_Z
\end{equation*}
absolutely converges and the sum does not depend on the choice of the basis $\{e_k\}$.
\end{prop}

\begin{defn}
Given $T \in \mathcal{L}_1(Z,Z)$ and an orthonormal basis $\{e_k\}_{k \in \mathbb N}$ of $Z$,
\begin{equation*}
    Tr(T) =  \sum_{k=1}^{+\infty} \braket{Te_k, e_k}_Z
\end{equation*}
is called the \emph{trace} of $T$.
\end{defn}

\noindent We have $\abs{Tr(T)} \le \norma{T}_{\mathcal{L}_1(Z,Z)}$.

\begin{defn}
Let $\{e_k\}_{k \in \mathbb N}$ be an orthonormal basis of $U$. The space of \emph{Hilbert–Schmidt operators} $\mathcal{L}_2(U,Z)$ from $U$ to $Z$ is defined by 
\begin{equation}\label{def:HS_operators}
    \mathcal{L}_2(U,Z) = \left\{ T \in \mathcal{L}(U,Z) : \sum_{k=1}^{+\infty} \norma{Te_k}_Z^2 < + \infty \right\}.
\end{equation}
\end{defn}

\begin{prop}
Let $\{e_k\}_{k \in \mathbb N}$ be an orthonormal basis of $U$. The space of Hilbert–Schmidt operators $\mathcal{L}_2(U,Z)$ does not depend on the choice of the orthonormal basis $\{e_k\}$. It is a separable Hilbert space if endowed with the inner product
\begin{equation*}
    \braket{S,T}_2 = \sum_{k=1}^{+\infty} \braket{Se_k, Te_k}_Z, \quad S, T \in \mathcal{L}_2(U,Z)
\end{equation*}
and the norm
\begin{equation*}
    \|T\|_{\mathcal{L}_2(U,Z)} = \sum_{k=1}^{+\infty} \|Te_k\|_Z, \quad T \in \mathcal{L}_2(U,Z).
\end{equation*}
The inner product and the norm are independent of the choice of the basis.
\end{prop}

\noindent In the sequel we need the following convergence result, see  \cite[Theorem 9.1.14]{hytonen2018analysis}.
\begin{thm}
\label{convergenceHilbertSchmidt}
Let $T \in \mathcal{L}_2(U,Z)$ and $\{L_n\}$ be a sequence of operators in $\mathcal{L}(Z,Z)$, such that
\begin{equation}
    L_n(x) \to L(x) \quad \mbox{for all} \quad x \in Z.
\end{equation}
\end{thm}
\noindent Then
\begin{equation}
L_n T \to L T \quad \mbox{in} \quad \mathcal{L}_2(U,Z).   
\end{equation}
Let $Q$ be a symmetric nonnegative definite trace-class operator on $U$ with non zero eigenvalues, i.e. $Q$ satisfies Definition \ref{traceclass} and 
$$
\langle Q u, v\rangle_U \geq 0 \quad \langle Q u, v\rangle_U = \langle u, Q v\rangle_U \quad \forall \; u,v \in U.
$$
The space $Q^{\frac{1}{2}} U$ is a separable Hilbert space equipped with the scalar product 
\begin{equation}
\braket{u, v}_{Q^{\frac{1}{2}} U} = \sum_{k=1}^{+\infty} \displaystyle\frac{1}{\lambda_k} \braket{u,e_k}_U \braket{v,e_k}_U,
\end{equation}
where $\{e_k\}_{k \in \mathbb{N}}$ is an orthonormal basis of $U$ and $\{\lambda_k\}_{k \in \mathbb{N}}$ is the sequence of eigenvalues of $Q$. 

\noindent We denote by $\mathcal{L}^0_2=\mathcal L_2(Q^{\frac{1}{2}} U,Z)$ the space of Hilbert Schimidt operators from $Q^{\frac{1}{2}}U$ to $Z$.

\subsection{Multivalued Mappings}

In this subsection, following 
\cite{koz}, we introduce some basic facts about multivalued analysis. Let $X$ and $Y$ be locally convex topological spaces.

\begin{defn}
A multivalued map (multimap for short) $F: X \to P(Y)$ is \emph{upper semicontinuous at the point} $x \in X$ if, for every open set $W \subseteq Y$ such that $F(x) \subset W$, there exists a neighborhood $V(x)$ of $x$ with the property that $F(V(x)) \subset W$. A multimap is \emph{upper semicontinuous} (shortly u.s.c.) if it is upper semicontinuous at every point $x \in X$.
\end{defn}

\begin{defn}
A multimap $F: X \to P(Y)$ is \emph{lower semicontinuous} at the point $x \in X$ if, for every open set $W \subseteq Y$ such that $F(x) \cap W \neq \emptyset$, there exists a neighborhood $V(x)$ of $x$ with the property that $F(x^\prime) \cap W \neq \emptyset$ for all $x^\prime \in V(x)$. A multimap is \emph{lower semicontinuous} (shortly l.s.c.) if it is lower semicontinuous at every point $x \in X$.
\end{defn}


\begin{defn}
A multimap $F: X \to P(Y)$ is said to be \emph{closed} if its graph $\Gamma_F$ is a closed subset of the space $X \times Y$.
\end{defn}
\begin{defn}
A multimap $F: X \to P(Y)$ is
\begin{itemize}
    \item[(a)] \emph{compact} if its range $F(x)$ is relatively compact in $Y$, i.e. $\overline{F(X)}$ is compact in $Y$;
    \item[(b)] \emph{locally compact} if every point $x \in X$ has a neighborhood $V(x)$ such that the restriction of $F$ to $V(x)$ is compact;
    \item[(c)] \emph{quasi compact} if its restriction to any compact subset $A \subset X$ is compact.
\end{itemize}
\end{defn}
\begin{thm}
    Let $F:X \to P(Y)$ be a closed locally compact multimap with compact values. Then, $F$ is u.s.c..
    \end{thm}
\noindent We also recall the classical notion of convexity for multivalued maps, see \cite{nikodem1987midpoint}.
    \begin{defn}
        A multimap $F: X \to P(Y)$ is said to be convex, if for every finite sequences $\{x_i\}_{i=1}^n \subset X$ and $\{\lambda_i\}_{i=1}^n$, such that $\lambda_i \geq 0$ for every $i=1,\dots,n$ and $\sum_{i=1}^n \lambda_i=1$ it follows that
        \begin{equation}
            \sum_{i=1}^n \lambda _i F(x_i) \subset F\left(\sum_{i=1}^n \lambda _i x_i \right).
        \end{equation}
    \end{defn}
\noindent Let $(E,\mathcal{E})$ be a complete and $\sigma$-finite measurable space and $X$ a Banach space with $X^*$ its topological dual.

\begin{defn}
A multimap $F: E \to P(X)$ is said to be \emph{measurable} if for every open subset $W \subset X$ the set $F^{-1}_-(W)$ is measurable.
\end{defn}

\begin{defn}
A map $f: E \to X$ is said to be a \emph{measurable selection} of a multimap $F: E \multimap X$, if $f$ is measurable and 
\begin{equation*}
    f(e) \in F(e) \,\, \text{for a.e.} \,\, e \in E.
\end{equation*}
\end{defn}

\begin{defn}
A map $f: E \to X$ is said to be 
\begin{itemize}
    \item[(a)] \emph{strongly measurable} if there exists a sequence of step functions convergent to $f$ a.e. on $E$;
    \item[(b)] \emph{scalarly measurable} if for every $x^* \in X^*$, $x^*(f):E \to \mathbb{R}$ is measurable.
\end{itemize}
\end{defn}
\noindent It is well known that a measurable map is also scalarly measurable.

\noindent In the next section, in order to prove the existence of at least one measurable selection we need the following two classical theorems on measurable maps taking values on a separable Banach space $X$.

\begin{thm}[Kuratowski-Ryll Nardzewski, \cite{Kuratowski-1965}]
\label{Kuratowsky}
Let $F:E \to P(X)$ be a measurable multimap with closed values. Then, $F$ admits a $(\mathcal{E},\mathcal{B}(X))$ measurable selector $f$, where $\mathcal{B}(X)$ denotes the Borel $\sigma$-algebra of $X$.
\end{thm}

\begin{thm}[Pettis, \cite
{Diestel-1977}
]
\label{Pettis}
A map $x: E \to X$ is strongly measurable if and only if it is scalarly measurable.
\end{thm}
\noindent Moreover, the following relation between measurability and strong measurability holds, see \cite[Proposition 1.10, Chapter 2]{hu-papageorgiou}.
\begin{thm}
\label{mis}
A map $x: E \to X$ is strongly measurable if and only if it is measurable.
\end{thm}

\subsection{Stochastic Processes}

Let $Z$ be a real separable Hilbert space with the norm $\norma{\cdot}_Z$.
\begin{defn}
A filtration $\mathbb{F}=(\mathscr{F}_t)_{t \ge 0}$ in a complete probability space $(\Omega, \mathscr{F}, \mathbf{P})$ is an increasing family of sub-$\sigma$-fields of $\mathscr{F}$. A filtration
$\mathbb{F}$
is said to satisfy the \emph{usual conditions} with respect to $\mathbf P$ if
\begin{itemize}
    \item[(i)] it is right-continuous, i.e. if $\mathscr{F}_t = \cap_{\epsilon > 0} \mathscr{F}_{t+\epsilon}$ for every $t \ge 0$;
    \item[(ii)] it is complete, i.e $\mathscr{F}_0$ (and then $\mathscr{F}_t$ for every $t \ge 0$) contains all $\mathbf{P}$-null sets of $\mathscr{F}$. 
\end{itemize}
\end{defn}
\begin{defn}
    Let $\mathbb F$ be a filtration satisfying the usual conditions on a complete probability space $(\Omega, \mathscr{F}, \mathbf{P})$. Then, $(\Omega, \mathscr{F}, \mathbf{P};\mathbb F)$ is called a {\em filtered probability space}.
\end{defn}
\begin{defn}
Let $(\Omega, \mathscr{F}, \mathbf{P};\mathbb F)$ be a filtered probability space. A $Z$-valued \emph{stochastic process} $X = (X_t)_{t \ge 0}$ is a family of random variables on $(\Omega, \mathscr{F})$, taking values in the Hilbert space $Z$.
\end{defn}

\begin{defn}
Let $(\Omega, \mathscr{F}, \mathbf{P};\mathbb F)$ be a filtered probability space. The process $X = (X_t)_{t \ge 0}$ is said to be 
\begin{itemize}
    \item[(a)] \emph{adapted} to the filtration $\mathbb{F}$ (or simply $\mathbb F$-adapted), if for every $t \ge 0$ the random variable $X_t$ is $\mathscr{F}_t$-measurable;
    \item[(b)] \emph{progressively measurable} with respect to $\mathbb{F}$  (or simply $\mathbb F$-progressively measurable) if for every $t \ge 0$ the mapping $[0,t] \times \Omega \to Z$, $(t, \omega) \to X_t(\omega)$ is $\mathcal{B}([0,t]) \times \mathscr{F}_t$-measurable, where for a given set $A$, $\mathcal{B}(A)$ denotes the Borel $\sigma$-algebra of $A$;
    \item[(c)] $\mathbf P$-\emph{integrable} if $\mathbb{E}\norma{X_t}_Z < + \infty$ for all $t \ge 0$;
    \item[(d)] a $(\mathbb F,\mathbf P)$-\emph{martingale} if it is $\mathbf P$-integrable, $\mathbb F$-adapted and $\mathbb{E}[X_t | \mathscr{F}_s] = X_s \quad \mathbf{P}$-a.s., for every $0 \le s \le t$.
    \end{itemize}
\end{defn}
\begin{defn}
Let $(\Omega, \mathscr{F}, \mathbf{P};\mathbb F)$ be a filtered probability space. Given two processes $X = (X_t)_{t \ge 0}$ and $Y = (Y_t)_{t \ge 0}$, we say that $X$ is a \emph{modification} of $Y$ if $X_t=Y_t$ $\mathbf{P}$-a.s. for every $t \geq 0$. We say that $X$ and $Y$ are \emph{indistinguishable} if, for almost all $\omega \in \Omega$, we have
\begin{equation}
    X_t(\omega)=Y_t(\omega), \quad \forall \; t \geq 0.
\end{equation}
\end{defn}
\begin{defn}
Let $(\Omega, \mathscr{F}, \mathbf{P};\mathbb F)$ be a filtered probability space. A random variable $\tau:\Omega \to [0,+\infty]$ is said to be a $\mathbb{F}$-\emph{stopping time} if for every $t \ge 0$, $\{ w \in \Omega : \tau (w) \le t \} \in \mathscr{F}_t$. For a process $X = (X_t)_{t \ge 0}$ and a $\mathbb F$-stopping time $\tau$, we denote by $X^\tau = (X_t^\tau)_{t \ge 0}$ the \emph{stopped process} of $X$ by $\tau$ defined as
\begin{equation*}
    X_t^\tau = X_{t \wedge \tau} =  X_t \mathbbm{1}_{\{ t <\tau\}} + X_\tau \mathbbm{1}_{\{t \ge \tau\}}, \quad t \ge 0.
\end{equation*}
\end{defn}

\begin{defn}
Let $(\Omega, \mathscr{F}, \mathbf{P};\mathbb F)$ be a filtered probability space. A $\mathbb F$-adapted process $X = (X_t)_{t \ge 0}$ with values in $Z$ is said to be a $(\mathbb F, \mathbf P)$-\emph{local martingale} if there exists an increasing sequence of stopping times $(\tau_n)_{n \in \mathbb{N}}$  such that
\begin{itemize}
    \item[(i)] $\tau_n \to + \infty $ as $n \to \infty$ $\mathbf P$-a.s.;
    \item[(ii)] the process $(X_t^{\tau_n})_{t \ge 0}$ is a $(\mathbb{F},\mathbf P)$-martingale, for every $n \in \mathbb N$.
\end{itemize}
We say that a sequence $(\tau_n)_{n \in \mathbb{N}}$ as above \emph{reduces} the $(\mathbb{F},\mathbf P)$-local martingale $X$.
\end{defn}
\noindent We recall that a standard one-dimensional Brownian motion is a real valued stochastic process $(\beta_t)_{t \geq 0}$ satisfying the following properties
\begin{itemize}
    \item[(i)] $\beta_0=0$;
    \item[(ii)] $\beta_t-\beta_s$ is independent of $\{\beta_r,\; r \in [0,s]\}$ for every $0 \leq s < t < +\infty$;
    \item[(iii)] $\beta_t-\beta_s$ is Gaussian with variance $(t-s)$ for every $0 \leq s < t < +\infty$.
\end{itemize}
\noindent Let $Q$ be a nonnegative trace class operator on a Hilbert space $U$.
\begin{defn}
Let $\{e_k\}_{k \in \mathbb N}$ be an orthonormal basis of $U$, $\{\lambda_k\}_{k \in \mathbb{N}}$ be the sequence of eigenvalues of $Q$ and $(\beta^k)_{k \in \mathbb N}$ be a sequence of mutually independent, standard one-dimensional Brownian motions $\beta^k: [0,+\infty) \times \Omega \to \mathbb R$ on $[0,+\infty)$. For every $t \in [0,+\infty)$, we set
 \begin{equation}
     W_t^Q:=\sum_{k=1}^{+\infty} \sqrt{\lambda_k}\, e_k \, \beta_t^k.
 \end{equation}
 The process $W^Q=(W_t^Q)_{t \ge 0}$ is called a $Q$-Wiener process on $[0,+\infty)$. When $Q$ is the identity on $U$, we will call it a {\em cylindrical} Wiener process in $U$.
 \end{defn}
 \noindent Finally, we recall some basic properties of the stochastic integral which will be useful in the sequel.
 \begin{prop} 
 \label{itoisometry}
 Let $T > 0$, $W^Q=(W_t^Q)_{t \ge 0}$ a $Q$-Wiener process on $[0,T]$ and $\Phi=(\Phi_t)_{t \in [0,T]}$ be a $\mathbb F$-progressively measurable process with values in $\mathcal{L}^0_2$ such that
 $$
\left(\mathbb{E}\displaystyle\int_0^T \|\Phi_s\|^2_{\mathcal{L}^0_2} \, ds\right)^{1/2} < \infty.
 $$
 Then,
 \begin{align}
\mathbb{E}\left(\displaystyle\int_0^t \|\Phi_s\|_{\mathcal{L}^0_2} \, dW_s^Q \right)^{2}& =\mathbb{E}\left(\displaystyle\int_0^t \|\Phi_s\|_{\mathcal{L}^0_2} \, ds\right)\\
 & =\mathbb{E}\left(\displaystyle\int_0^t Tr[(\Phi_sQ^{1/2})(\Phi_sQ^{1/2})^*]\,ds\right) \quad t \in [0,T],
 \end{align}
 where given an operator $L$, $L^*$ denotes its adjoint operator.
 \end{prop}
 \noindent Let $T > 0$, $W^Q=(W_t^Q)_{t \ge 0}$ a $Q$-Wiener process on $[0,T]$ and $\Phi=(\Phi_t)_{t \in [0,T]}$ be a $\mathcal{L}^0_2$-valued process stochastically integrable, $\varphi=(\varphi_t)_{t \in [0,T]}$ a $Z$-valued progressively measurable with respect to $\mathbb{F}$ and Bochner integrable $\mathbf{P}$-a.s. on $[0,T]$ process and $X_0$ a $\mathscr{F}_0$-measurable $Z$-valued random variable. Then, the following process
 $$
 X_t=X_0+\displaystyle\int_0^t \varphi_s\,ds + \int_0^t \Phi_s\, dW_s^Q, \quad t \in [0,T],
 $$
 is well defined, moreover the following classical result holds.
 \begin{prop}
 \label{itoformula}
     Assume that a function $F:[0,T] \to \mathbb{R}$ and its partial derivatives $F_t$, $F_x$ and $F_{xx}$ are uniformly continuous on bounded subsets of $[0,T]\times H$, then $\mathbf P$-a.s., for all $t \in [0,T]$
     \begin{align}
         F(t,X_t)= & F(0,X_0)+\int_0^t F_x(s,X_s) \Phi_s \, dW_s^Q \\
         & + \int_0^t F_t(s,X_s)+\langle F_x(s,X_s), \varphi_s \rangle \, ds \\
         & +\frac{1}{2} Tr [F_xx(s,X_s)(\Phi_sQ^{1/2})(\Phi_s Q^{1/2})^*]\, ds.
     \end{align}
 \end{prop}
\noindent To simplify the notation, in what follows we denote the $Q$-Wiener process $W^{Q}$ simply by $W$.
 
\section{Stochastic Evolution Inclusions} \label{sec:SEI}

\noindent In this section we present an existence result of mild solutions for stochastic evolution inclusions. Throughout the paper, $(\Omega, \mathscr{F},\mathbf P)$ denotes a complete probability space endowed with a filtration $\mathbb{F}=(\mathscr{F}_t)_{t \ge 0}$ satisfying the usual conditions.
All random elements will be defined on this stochastic basis without further notice. We also stipulate that random variables equal outside an event of probability zero are considered equal, and that equality of stochastic processes is intended in the sense of indistinguishability. We shall denote by $W=(W_t^Q)_{t \geq 0}$ a given $K$-valued $Q$-Wiener process, with a symmetric nonnegative trace class operator $Q$, where $K$ is a separable Hilbert space with norm $\norma{\, \cdot \,}_K$ and inner product $ \langle \cdot,\cdot\rangle_K$.
Recall that $\mathcal L_2^0=\mathcal L_2(Q^{\frac{1}{2}}K,H)$ is the Hilbert space of all Hilbert-Schmidt operators from $Q^{\frac{1}{2}}K$ to $H$, where $H$ is an additional separable Hilbert space.

\noindent We are interested in finding solutions of the problem \eqref{inclusione} in the function space $C([0,+\infty), L^2(\Omega, H))$, the normed space of continuous functions from $[0,+\infty)$ into $L^2(\Omega, H)$, satisfying 
\begin{align}
 &\sup_{t \in [0,+\infty)} \norma{u_t}^2_{L^2(\Omega, H)} < \infty. 
 \end{align}
Here, $L^2(\Omega, H)$ is the Banach space of all square integrable random variables with the norm 
$$
\norma {v}_{L^2(\Omega, H)}= ( \mathbb{E} \norma{v}^2 _H)^{\frac{1}{2}},
$$
where $\mathbb{E} \norma{v}^2 _H = \int_{\Omega} \norma{v(\omega)}^2 _H d\mathbf P_{\omega}$. In particular, we look for solutions in $\mathcal{C}$, a closed subspace of $C([0,+\infty), L^2(\Omega, H))$ consisting of all $\mathbb{F}$-adapted, $H$-valued processes $u \in C([0,+\infty), L^2(\Omega, H))$ endowed with the norm
\begin{equation*}
\norma u_{\mathcal{C}} = \sup_{t \in [0,+\infty)} \norma{u_t}_{L^2(\Omega, H)} =  \sup_{t \in [0,+\infty)} \Big( \mathbb{E} (\norma{u_t}^2_H) \Big)^{\frac{1}{2}}.
\end{equation*}

\noindent Our aim is to study the existence of mild solutions for the stochastic evolution inclusion \eqref{inclusione} in $\mathcal{C}$, by means of a weak topology approach. More precisely, denoting by $L^2_{\mathscr{F}}([0,+\infty),\mathcal L_2^0)$ the space of all $\mathbb{F}$-progressively measurable processes defined on $[0,+\infty)$ with values in $\mathcal L_2^0$ and the norm
\begin{equation}
\|\sigma\|_{L^2_{\mathscr{F}}}:=\int_0^{+\infty} \mathbb E(\|\sigma_t\|_{\mathcal L_2^0}^2) \,dt < \infty,  
\end{equation}
we look for solutions according to the following definitions, see \cite[Definitions 1.118, 1.119]{fabbrigozzi}.  

\begin{defn}
\label{mild}
A $\mathbb F$-progressively measurable stochastic process $u \in \mathcal{C}$ is said to be a {\em mild solution} of inclusion (\ref{inclusione}) if $u(0)=u_0$ and there exists $\sigma \in L^2_{\mathscr{F}}([0,+\infty),\mathcal L_2^0)$ such that $\sigma_t \in \Sigma(t, u_t)$ for a.e. $t \in [0,+\infty)$ satisfying the following integral equation
\begin{equation}
\label{eqmild}
u_t = T(t)u_0 + \int_0^t T(t-s)f(s, u_s)\,ds + \int_0^t T(t-s)\sigma_s\,dW_s, \,\, \text{for} \,\, t \in [0,+\infty).
\end{equation}
\end{defn}
\begin{defn}
\label{strong}
A $\mathbb F$-progressively measurable stochastic process $u \in \mathcal{C}$ is said to be a {\em strong solution} of inclusion (\ref{inclusione}) if
\begin{itemize}
    \item[(i)] $u_t \in D(A)$, $dt \otimes d\mathbf P$-a.e., $\displaystyle\int_0^\tau \|Au_t\|_H \, dt < \infty$ $\mathbf P$-a.s., for every $\tau \in [0,+\infty)$;
    \item[(ii)] there exists $\sigma \in L_{\mathscr{F}}^2([0,+\infty), \mathcal L_2^0)$ such that $\sigma_t \in \Sigma(t, u_t)$ for a.e. $t \in [0,+\infty)$ satisfying the following integral equation
    $$
    u_t = u_0 + \int_0^t (A u_s+f(s, u_s))\,ds + \int_0^t \sigma_s\,dW_s, \,\, \text{for} \,\, t \in [0,+\infty).
    $$
\end{itemize}
\end{defn}
\noindent A strong solution is also a mild solution, to have the inverse implication we need to assume additional hypotheses on the operator $A$, the map $f$ and the multivalued map $\Sigma$ as we will show in the next section, following the proof of \cite[Theorem 2]{albeverio2017ito}.

\subsection{Statement of the Problem}

We study the stochastic evolution inclusion \eqref{inclusione} under the following assumptions:
\begin{itemize}
    \item[($H_A$)] the operator $A$ generates a strongly continuous semigroup $\{T(t)\}_{t\ge 0}$ in $H$;
    \item[($H_1$)] the function $f(\cdot, u): [0,+\infty) \to \spL$ is measurable for every $u \in H$; 
    \item[($H_2$)] the function $f(t, \cdot) : \spL \to \spL$ is linear and continuous for each $t \in [0,+\infty)$;
    \item[($H_3$)] there exists a function $C_f \in L^1([0,+\infty),\mathbb{R}_+)$ such that 
    \begin{equation}
      \norma{f(t,u)}_H^2  \le C_f(t)(1 +\norma{u}_\spL^2)  
    \end{equation}
    for all $t \in [0,+\infty)$ and $u \in H$;
\end{itemize}
The multimap $\Sigma : [0,+\infty)  \times \spL \to P(\mathcal L_2^0)$ has closed bounded and convex values, and satisfies the following conditions:
\begin{itemize}
    \item[($H_4$)] $\Sigma : [0,+\infty)  \times \spL \to P(\mathcal L_2^0)$ admits a $(\mathcal{B}([0,t]) \times \mathcal{B}(H)$, $\mathcal{B}(\mathcal L_2^0))$ jointly measurable selection for every $t \geq 0$, i.e. there exists a map $\sigma: [0,+\infty) \times H \to \mathcal L_2^0$ such that is $(\mathcal{B}([0,t]) \times \mathcal{B}(H)$,$\mathcal{B}(\mathcal L_2^0))$ jointly measurable for every $t \geq 0$ and $\sigma(t,x) \in \Sigma(t,x)$ for a.e. $(t,x) \in [0,+\infty) \times H$;
    
    \item[($H_5$)] $\Sigma(t, \cdot) : \spL \multimap P(\mathcal L_2^0) $ is convex and weakly sequentially closed for a.e. $t \in [0,+\infty)$, i.e. it has a weakly sequentially closed graph;
    \item[($H_6$)] there exists a function $C_\Sigma \in L^1([0,+\infty),\mathbb{R}_+)$ such that 
    \begin{equation}
\norma{\Sigma(t,u)}_{\mathcal L_2^0}^2  \le C_\Sigma(t) (1 + \norma{u}_\spL^2) 
    \end{equation}
    for all $t \in [0,+\infty)$ and $u \in H$,
    where 
    \begin{equation}
       \norma {\Sigma(t, u)}_{\mathcal L_2^0} = \sup \{\norma{\sigma_t}_{\mathcal L_2^0} : \sigma_t \in \Sigma(t, u), \, \mbox{for a.e.} \, t \in [0,+\infty) \}.
    \end{equation}   
\end{itemize}
\noindent We will give now an example of a multivalued map that satisfies the assumptions $(H_4)$, $(H_5)$ and $(H_6)$, see \cite{benedetti2010semilinear}.
\begin{ex}
Let us consider the multimap $\Sigma : [0,+\infty)  \times \spL \to P(\mathcal L_2^0)$ defined as
\begin{equation}
\Sigma(t,u) = g(t) \langle\varphi, \, u \rangle_H G,
\end{equation}
where
\begin{itemize}
\item[(i)] $g \in {L^{\infty}([0,+\infty), \mathbb{R})}$;
\item[(ii)] $\varphi \in H$;
\item[(iii)] $G \subset \mathcal{L}^0_2$ is nonempty, convex, bounded and closed.
\end{itemize}
We prove that $\Sigma$ satisfies the assumptions $(H_4)$, $(H_5)$, $(H_6)$.

\noindent According to (iv), $\Sigma$ has nonempty, bounded, closed and convex values.

\noindent By the measurability of $g$ and the linearity of the scalar product, it is easy to show that $\Sigma$ is measurable and so by Theorem \ref{Kuratowsky} it admits a measurable selection. 

\noindent Again by the linearity of the scalar product, $\Sigma(t,\cdot)$ is convex for a.e. $t \in [0,+\infty)$ and the map $u\longmapsto \langle\varphi, \, u \rangle_H$ is continuous from
$H_{w}$ to $\mathbb{R}$, thus $\Sigma(t, \cdot) \,: H \to P(\mathcal{L}^0_2)$ is u.s.c. with respect to the weak topology, because it is the product of a weakly continuous single valued map with a constant multivalued map, hence it is weakly closed because it is u.s.c. with respect to the weak topology with weakly closed values.

\noindent By the boundedness of the set $G$, there exists a constant $N$ such that $\|G\|_{\mathcal L^0_2} \leq N$. Thus, for all $u \in H$ and a.a. $ t \in [0,+\infty)$, we have that 
\begin{align}
||\Sigma(t,u)||_{\mathcal L^0_2} & \le |g(t)| |\langle \varphi,u\rangle| \|G\|_{\mathcal L^0_2} \leq |g(t)| \|\varphi\|_H \|u\|_H N.
\end{align}
Hence, hypothesis $(H_6)$ is also fulfilled, with $C_\Sigma(t)= N |g(t)| \|\varphi\|_H$.
\end{ex}

\noindent We state now the main result of this section.
\begin{thm}
\label{th-main}
Assume that $(H_A),(H_1)-(H_6)$ hold. 
Then, given $u_0 \in L^2(\Omega,H)$, the inclusion \eqref{inclusione} has at least one mild solution on the whole half line $[0,+\infty)$, $u \in \mathcal{C}$.
\end{thm}
\begin{rem}
\label{mild-strong}
{\rm We notice that, under the assumptions of Theorem \ref{th-main}, every mild solution $u$ of \eqref{inclusione} such that $u_t \in D(A)$ for every $t \in [0,+\infty)$ is a strong solution of \eqref{inclusione}. Indeed, let $u$ be a mild solution of \eqref{inclusione}, from the assumptions for $t \in [0,+\infty)$ we have
\begin{equation}
\begin{split}
\displaystyle\int_0^t A u_s \, ds & = \int_0^t A T(s)u_0 \,ds + \int_0^t \int_0^s A T(s-r)f(r, u_r)\,dr\,ds \\
& \qquad + \int_0^t \int_0^s A T(s-r)\sigma_r\,dW_r\,ds.
\end{split}
\end{equation}
Hence, by applying the Fubini Theorem for Lebesgue integrals and the Fubini Theorem for stochastic integrals (see \cite[Theorem 4.33]{DaPrato}), we get
\begin{equation}
\begin{split}
\displaystyle\int_0^t A u_s \, ds & = \int_0^t A T(s)u_0 \,ds + \int_0^t \int_r^t A T(s-r)f(r, u_r)\,ds\,dr \\ & \qquad + \int_0^t \int_r^t A T(s-r)\sigma_r\,ds \, dW_r.
\end{split}
\end{equation}
Now, applying the formula
$$
\int_0^t A T(s) \xi \,ds=T(t)\xi-\xi, \quad \forall \; \xi \in D(A),
$$
we get that $Au_t$ is integrable with probability one and
\begin{equation}
\begin{split}
\displaystyle\int_0^t A u_r \, dr &  = T(t)u_0-u_0+ \int_0^t T(t-r)f(r, u_r)\,dr -\int_0^t f(r,u_r)\,dr \\
& \qquad + \int_0^t T(t-r)\sigma_r\,dW_r
 -\int_0^t \sigma_r \,dWr.
\end{split}
\end{equation}
Thus, obtaining
$$
u_t=u_0+\int_0^t A u_r \, dr + \int_0^t f(r,u_r)\,dr + \int_0^t \sigma_r \,dWr.
$$
By Definition \ref{strong} $u$ is a strong solution of \eqref{inclusione}.
}
\end{rem}
\noindent Moreover, we notice that in the case of a semigroup of contraction generated by the linear part $A$, under the assumptions $(H_3)$ and $(H_6)$ it is possible to obtain the global boundedness of the solution of \eqref{inclusione} as stated in the following proposition.
\begin{prop}
\label{boundsol}
 Assume that $A$ generates a semigroup of contraction and that $(H_1)-(H_6)$ hold. Then, if $u_0 \in L^2(\Omega,H)$, there exists a constant $\overline{C} > 0$ such that
 \begin{equation}
     \mathbb{E}\norma{u_t}_H^2 \leq \overline{C} \quad \mbox{for every} \; t \in [0,+\infty),
 \end{equation}
 where $u\in \mathcal{C}$ is a mild solution of \eqref{inclusione}.
\end{prop}
\begin{proof}
    Let $u \in \mathcal{C}$ be a mild solution of \eqref{inclusione}. Then, we have that there exists a process $\sigma \in L^2([0,+\infty),\mathcal L_2^0)$, with $\sigma_t \in \Sigma(t, u_t)$ for a.e. $t \in [0,+\infty)$ such that
    \begin{align}
        \mathbb{E}\norma{u_t}_H^2 & \leq  3 \mathbb{E}\norma{u_0}^2_H + 3 \displaystyle\int_0^t \mathbb{E}\norma{f(s,u_s)}^2_H \,ds + 3 {\rm tr}(Q)\displaystyle\int_0^t \mathbb{E}\norma{\sigma_s}_{\mathcal L_2^0}^2\,ds,
    \end{align}
    for every $t \in [0,+\infty)$. By assumptions $(H_3)$ and $(H_6)$ we have that
    \begin{align}
      \mathbb{E}\norma{u_t}_H^2 \leq & 3\mathbb{E}\norma{u_0}_H^2 +3 \|C_f\|_{L^1([0,+\infty), \mathbb{R})}+3{\rm tr}(Q)\norma{C_\Sigma}_{L^1([0,+\infty), \mathbb{R})} \\
      & + 3 \displaystyle\int_0^t (C_f(s)+{\rm tr}(Q) C_\Sigma(s)) \mathbb{E}\norma{u_s}_H^2 \,ds.
    \end{align}
    By the Gronwall Lemma for every $t \in [0,+\infty)$ it follows that 
\begin{align}
   \mathbb{E}\norma{u_t}_H^2 & \leq C e^{3 \int_0^t (C_f(s)+{\rm tr}(Q) C_\Sigma(s))\,ds} \\
   & \leq C e^{3 \left(\|C_f\|_{L^1([0,+\infty), \mathbb{R})}+{\rm tr}(Q)\norma{C_\Sigma}_{L^1([0,+\infty), \mathbb{R})}\right)} = : \overline{C}, 
\end{align}   
where we have set $C=3\left(\mathbb{E}\norma{u_0}_H^2 + \|C_f\|_{L^1([0,+\infty), \mathbb{R})}+{\rm tr}(Q)\norma{C_\Sigma}_{L^1([0,+\infty), \mathbb{R})}\right)$.
\end{proof}
\begin{rem}
    In order to obtain only a local existence result, that is there exists a maximal interval $[0,t_M[$ such that \eqref{inclusione} has a mild solution for every $t \in [0,t_M[$, with $t_M < +\infty$, it is possible to weaken the assumptions $(H_3)$ and $(H_6)$ as follows
    \begin{itemize}
    \item[$(H_3)^\prime$] for every $r>0$, there exists a function $\ell_r \in L^1([0,+\infty), \mathbb{R}^{+})$ such that for each $u \in \spL$ with $\|u\|_H^2 \le r$, 
        \begin{equation*}
        \norma {f(t,u)}_{H}^2 \le \ell_r(t) \quad \text{for a.e.} \quad t \in [0,+\infty);
        \end{equation*}
        \item[$(H_6)^\prime$] for every $r>0$, there exists a function $\mu_r \in L^1([0,+\infty), \mathbb{R}^{+})$ such that for each $u \in \spL$,with $\|u\|_H^2 \le r$, 
        \begin{equation*}
        \norma {\Sigma(t, u)}^2_{\mathcal L_2^0} \le \mu_r(t) \quad \text{for a.e.} \quad t \in [0,+\infty).
        \end{equation*}
        \end{itemize}
\end{rem}
\noindent To establish the existence of solutions for the problem \eqref{inclusione} on the half-line $[0,+\infty)$, we employ an appropriate approximation technique that relies on truncating the problem \eqref{inclusione} to the bounded interval $[0,k]$ where $k > 0$.

\section{Approximating problems}\label{sec:approx}

Let $k \in \mathbb{N}$, $k > 0$. We consider \eqref{inclusione} restricted to the bounded interval $[0,k]$:
\begin{equation}
\label{inclusione-approssimata}
\begin{cases}
du_t \in [Au_t + f(t, u_t)]dt + \Sigma(t,u_t)dW_t, \quad t \in [0,k] \\
u(0)=u_0,
\end{cases}
\end{equation}
under assumptions $(H_A)$, $(H_1)-(H_6)$. 

\noindent Let us denote by $\mathcal{C}^k=\{u\big{|}_{[0,k]} : u \in \mathcal{C}\}$. Clearly, the space $\mathcal{C}^k$ is endowed with the norm of $\mathcal{C}$ restricted on $[0,k]$.

\noindent With the next result, we provide sufficient conditions for the existence of a selection of the multivalued map $\Sigma(\cdot,u_\cdot)$ for any $u \in \mathcal{C}^k$.
\begin{lem}
\label{zhou lemma 8.1}
Assume that the multimap $\Sigma$ satisfies conditions $(H_4), (H_5), (H_6)$. Then, for every $k \in \mathbb{N}$ the set 
\begin{equation*}
Sel_{\Sigma}^k(u) = \{\sigma \in L^2_{\mathscr{F}}([0,k], \mathcal L_2^0) : \sigma_t \in \Sigma(t, u_t) \quad \text{for a.e.} \, t \in [0,k] \}.
\end{equation*} 
is nonempty for any $u \in \mathcal{C}^k$.
\end{lem}
\begin{proof}
Let $u \in \mathcal{C}^k$. By assumptions, $u$ is $\mathbb{F}$-adapted, and has continuous trajectories, moreover $H$ is a separable Hilbert space, thus $u$ has a $\mathbb{F}$-progressively measurable modification, still denoted by $u$, (see \cite[Proposition 3.6]{DaPrato}. We denote with $g:[0,k] \times \Omega \to \mathcal{L}_2^0$ the map defined as
\begin{equation}
    g(t,\omega)=\sigma(t,u(t,\omega)),\quad (t,\omega) \in [0,k] \times \Omega,
\end{equation}
where $\sigma$ is the selection of the multimap $\Sigma$ from assumption $(H_4)$. Notice that by construction
\begin{equation}
\label{sel}
    g(t,\omega)=\sigma(t,u(t,\omega)) \in \Sigma(t,u(t,\omega)) \quad \mbox{for a.e.} \; (t,\omega) \in [0,k] \times \Omega.
\end{equation}
Now, let $U \subset \mathcal L_2^0$ be an open set. We will prove that for every $t \in [0,k]$, $g^{-1}(U)$ is $\mathcal{B}([0,t]) \times \mathscr{F}_t$ measurable, where we recall that
\begin{equation}
    g^{-1}(U)=\{(t,\omega) \in [0,k] \times \Omega: \; g(t,\omega)=\sigma(t,u(t,\omega)) \in U \}.
\end{equation}
Notice that 
\begin{equation}
    g^{-1}(U)=\{(t,\omega) \in [0,k] \times \Omega: \; \exists \; (t,x) \in D: \ u(t,\omega)=x\}
\end{equation}
with
\begin{equation}
    D=\{(t,x) \in [0,k] \times H:  \; \sigma(t,x) \in U \}=\sigma^{-1}(U).
\end{equation}
Since $D=D_1 \times D_2 \subset [0,k] \times H$, it follows that
\begin{align}
    g^{-1}(U) & =\{(t,\omega) \in [0,k] \times \Omega:  \; \exists \; t \in D_1 : u(t,\omega) \in D_2\} \\
    & = \{(t, \omega) \in D_1 \times \Omega: \; (t,\omega) \in u^{-1}(D_2)\}.
\end{align}
By assumption $(H_4)$, for every $t \in [0,k]$, $D \in \mathcal{B}([0,t]) \times \mathcal{B}(H)$ and so by the $\mathbb F$-progressive measurability of $u$ we have that, for every $t \in [0,k]$, $u^{-1}(D_2) \in \mathcal{B}([0,t]) \times \mathscr{F}_t$. In conclusion, for every $t \in [0,k]$, $g^{-1}(U)\in \mathcal{B}([0,t]) \times \mathscr{F}_t$, proving that $g$ is $\mathbb{F}$-progressively measurable. Finally, by $(H_6)$, $g \in L^2([0,k] \times \Omega, \mathcal L_2^0)$. Thus, taking \eqref{sel} into account, $g \in Sel_\Sigma^k(u)$ concluding the proof.
\end{proof}
\begin{rem}
    For every $u \in \mathcal{C}^k$ the map $h:[0,k] \times \Omega \to H$ defined as the composition $h(t,\omega)=f(t,u(t,\omega))$ is separately strongly measurable. Indeed, by the classical Carath\'eodory Theorem $u$ is jointly measurable. Moreover, for a.e. $\omega \in \Omega$ the map $h(\cdot,\omega)$ is $(\mathcal{B}({0,k}),\mathcal{B}(H))$-measurable as composition of measurable maps and it is strongly measurable by Theorem \ref{mis} and the separability of the space $H$. On the other hand, for a.e. $t \in [0,k]$ the map $f(t,\cdot)$ is scalarly measurable being weakly continuous, thus by Theorem \ref{Pettis} and the separability of the space $H$, it is strongly measurable and so the map $h(t,\cdot)$ is strongly measurable as composition of strongly measurable maps. 
\end{rem}

\noindent We also recall that, for a strongly continuous semigroup, there exists a constant $M_1 \ge 1$ such that
\begin{equation}
\label{M1}
\sup_{t \in [0,k]} \norma{T(t)}_{\mathcal{L}(H)} \le M_1.
\end{equation} 

\noindent We introduce the multioperator $\mathscr{F}_k : \mathcal{C}^k \multimap \mathcal{C}^k$ as follows:
\begin{equation}\label{eq:multiop}
    \mathscr{F}_k(u) = S \circ Sel_{\Sigma}^k(u), \quad u \in \mathcal{C}^k,
\end{equation}
where $S: L^2_{\mathscr{F}}([0,k], \mathcal L_2^0) \to C([0,k], L^2(\Omega, H))$ is defined as 
\begin{equation*}
    S(\sigma)_t = T(t)u_0 + \int_0^t T(t-s)f(s, u_s) \, ds + \int_0^t T(t-s)\sigma_s \, dW_s, \quad t \in [0,k]
\end{equation*}
whose fixed points are mild solutions of inclusion \eqref{inclusione-approssimata}. 

\noindent Notice that the operator $\mathscr{F}_k$ is well defined since for $\sigma \in \mathcal{C}^k$, $S(\sigma) \in \mathcal{C}^k$ as well, because it is an $\mathbb{F}$-adapted process in view of the measurability of $f$, $\sigma$ and the properties of the Wiener process $W$. Denoting by $C_1 > 0$ the following constant,
\begin{equation}
\label{C1}
    C_1:= 3 M_1^2(\mathbb{E} \norma{u_0}_H^2 + \|C_f\|_{L^1([0,k], \mathbb{R})}+{\rm tr}(Q)\norma{C_\Sigma}_{L^1([0,k], \mathbb{R})}),
\end{equation}
\noindent we consider the set 
\begin{equation}
\mathscr{Q} =  \{ u\in \mathcal{C}^k : \mathbb{E}\norma{u}_H^2 \le R e^{Lt} \; \mbox{a.e.} \; t \in [0,k]\}
\end{equation}
with $L > 0$ and $R > 0$ such that
\begin{equation}
\label{RL}
\begin{array}{l}
    q:= \displaystyle\max_{t \in [0,k]} 3 M_1^2 \displaystyle\int_0^t (C_f(s)+{\rm tr}(Q) C_\Sigma(s)) e^{L(s-t)} \, ds < 1 \\
    \mbox{and}\\
    R \geq \displaystyle\frac{C_1}{1-q}.
    \end{array}
\end{equation}

\noindent The proof of the existence of at least one mild solution of \eqref{inclusione-approssimata} is based on the following fixed point theorem.

\begin{thm}[Theorem 2.2 in \cite{o2000fixed}]
\label{zhou th 1.29}
Let $X$ be a metrizable locally convex linear topological space and let $D$ be a weakly compact, convex subset of $X$. Suppose $\varphi : D \to P(D)$ is a multivalued map with closed and convex values and a weakly sequentially closed graph. Then $\varphi$ has a fixed point, i.e. there exists $x_0 \in D$ such that $x_0 \in \varphi(x_0)$.
\end{thm}

\begin{thm}
\label{zhou th 8.1}
Assume that $(H_A),(H_1)-(H_6)$, hold. Then, given $u_0 \in L^2(\Omega,H)$, the inclusion \eqref{inclusione-approssimata} has at least one mild solution on the interval $[0,k]$ and the solution set is a weakly compact set in $C([0,k],L^2(\Omega,H))$.
\end{thm}
\begin{proof}
We will apply Theorem \ref{zhou th 1.29} to the operator $\mathscr{F}_k$ introduced in \eqref{eq:multiop}. We divide the proof in the following steps.

\noindent {\bf Step 1.} {\it The multioperator  $\mathscr{F}_k$ has a weakly sequentially closed graph.}

\noindent Let $\{u^m\} \subset \mathscr{Q}$ and $\{v^m\} \subset \mathcal{C}^k$ be such that $v^m \in \mathscr{F}_k(u^m)$ for all $m$ and $u^m \rightharpoonup u$, $v^m \rightharpoonup v$ in $\mathcal{C}^k$, we will prove that $v \in \mathscr{F}_k(u)$. 


\noindent The fact that $v^m \in \mathscr{F}_k(u^m)$ means that there exists a sequence $\{\sigma_m\}$, $\sigma_m \in Sel_\Sigma^k (u^m)$, such that for every $t \in [0,k]$,
\begin{equation*}
v^m(t) = T(t)u_0 + \int_0^t T(t-s)f(s, u^m(s)) \, ds + \int_0^t T(t-s)\sigma^m(s) \, dW_s.
\end{equation*}
By the definition of the convergence in $\mathcal{C}^k$ we have that $u^m(t)$ weakly converges to $u(t)$ in $L^2(\Omega,H)$ for every $t \in [0,k]$, thus considering $\varphi \in L^2([0,k], L^2(\Omega, H))$, it follows that 
\begin{equation}
   \langle u^m(t)-u(t) , \varphi(t) \rangle_{L^2(\Omega,H)} \to 0 \quad \forall \; t \in [0,k].
\end{equation}
Moreover, 
\begin{equation}
\begin{array}{lcl}
 |\langle u^m(t)-u(t) , \varphi(t) \rangle_{L^2(\Omega,H)}| & \leq & \left(\mathbb{E}\norma {u^m(t)-u(t)}^2_{H} \mathbb{E}\norma{\varphi(t)}^2_H \right)^{1/2} \\
 & \le & \sqrt{2 R} e^{\frac{Lk}{2}} \left(\mathbb{E}\norma{\varphi(t)}^2_H\right)^{1/2}.  
 \end{array}
\end{equation}
Thus, by the Lebesgue dominated convergence Theorem we have that
\begin{align}
\displaystyle\int_0^k \langle u^m(s)-u(s) , \varphi(s) \rangle_{L^2(\Omega,H)}\, ds \to 0,
\end{align}
so, by the arbitrariness of $\varphi$ we get that $u^m \rightharpoonup u$ in $L^2([0,k], L^2(\Omega, H))$. Let $x' : L^2(\Omega, H) \to \mathbb{R}$ a linear and continuous operator. By the linearity and continuity of the integral, the evolution operator $\{T(t)\}_{t \ge 0}$, and $f(t,\cdot)$, we have that the operator
\begin{equation*}
g \mapsto x' \circ \int_0^t T(t-s)f(s,g(s)) \, ds 
\end{equation*}
\noindent is a linear and continuous operator from $L^2([0,k], L^2(\Omega, H))$ to $\mathbb{R}$ for all $t \in [0,k]$. Then, from the definition of the weak convergence, we have for every $t\in [0,k]$,
\begin{equation*}
x' \circ \int_0^t T(t-s)f(s, u^m(s)) \, ds  \rightarrow x' \circ \int_0^t T(t-s)f(s, u(s)) \, ds.
\end{equation*}
Thus,
\begin{equation*}
\int_0^t T(t-s)f(s, u^m(s)) \, ds  \rightharpoonup \int_0^t T(t-s)f(s, u(s)) \, ds.
\end{equation*}
\noindent Moreover, we observe that, since $u^m \in \mathscr{Q}$, according to $(H_6)$, 
\begin{equation}
  \mathbb{E}\norma {\sigma^m(t)}^2_{\mathcal L_2^0} \le C_\Sigma(t)(1+R e^{Lk})  
\end{equation}
for a.e. $t \in [0,k]$ and every $m$, i.e., $\{\sigma^m\}$ is uniformly integrable in $L^2([0,k] \times \Omega,\mathcal L_2^0)$. Hence, by the reflexivity of the space $\mathcal L_2^0$ and Theorem \ref{weakcompL1}, we have the existence of a subsequence, denoted as the sequence, and a function $\sigma$ such that $\sigma^m \rightharpoonup \sigma$ in $L^2([0,k] \times \Omega,\mathcal L_2^0)$.
Moreover, we have
\begin{equation}
\label{weak-conv}
\int_0^t T(t-s)\sigma^m(s) \, dW_s \rightharpoonup \int_0^t T(t-s)\sigma(s) \, dW_s \quad \forall \; t \in [0,k].
\end{equation}

\noindent We first prove that the operator
\begin{equation*}
h \mapsto \int_0^t T(t-s)h(s) \, dW_s
\end{equation*}

\noindent is linear and continuous from $L^2([0,k] \times \Omega, \mathcal L_2^0)$ to $L^2(\Omega, H)$.
Indeed, the linearity follows from the linearity of the semigroup and the integral operator.
Moreover for any $h^m, h \in L^2([0,k] \times \Omega, \mathcal L_2^0)$ and $h^m \rightarrow h$ as $m \rightarrow \infty$, by \eqref{M1}, we get for each $t \in [0,k]$,
\begin{align*}
& \mathbb{E} \norma*{ \int_0^t T(t-s)[h^m(s) - h(s)] \, dW_s}_{\mathcal L_2^0}^2 \\
&  \le M_1^2 {\rm tr}(Q)\int_0^t \mathbb{E}\norma{ h^m(s) - h(s)}_{\mathcal L_2^0}^2 \, ds \rightarrow 0, \quad \text{as} \quad m \rightarrow \infty.
\end{align*}

\noindent Hence, the operator
\begin{equation*}
h \mapsto \int_0^t T(t-s)h(s) \, dW_s
\end{equation*} 
is continuous. Thus, we have that the operator 
\begin{equation*}
h \mapsto x' \circ \int_0^t T(t-s)h(s) \, dW_s
\end{equation*}

\noindent is a linear and continuous operator from $L^2([0,k] \times \Omega, \mathcal L_2^0)$ to $\mathbb{R}$ for all $t \in [0,k]$. Then, from the definition of the weak convergence, we have for every $t \in [0,k]$,
\begin{equation*}
x' \circ \int_0^t T(t-s)\sigma^m(s) \, dW_s \rightarrow x' \circ \int_0^t T(t-s)\sigma(s) \, dW_s,
\end{equation*}
that is \eqref{weak-conv}.

\noindent From the above arguments, we have 
\begin{align*}
v^m(t) \rightharpoonup & \, T(t)u_0 + \int_0^t T(t-s)f(s, u(s)) \, ds \\
& + \int_0^t T(t-s)\sigma(s) \, dW_s = v^*(t), \quad \forall \, t \in [0,k],
\end{align*}

\noindent which implies, for the uniqueness of the weak limit in $L^2(\Omega, H)$, that $v^*(t)=v(t)$ for all $t\in [0,k]$.
Finally, to obtain the claimed result we have to prove that $\sigma(t) \in \Sigma(t, u(t))$ for a.e. $t \in [0,k]$. 

\noindent First of all notice that by $(H_6)$, the multimap $\Sigma(t, \cdot)$ is locally weakly compact for a.e. $t \in [0,k]$, i.e., for a.e. $t \in [0,k]$ and every $u \in \spL$, there is a neighbourhood $V$ of $u$ such that the restriction of $\Sigma(t, \cdot)$ to $V$ is weakly compact.
Hence, by $(H_5)$ and the locally weak compactness, we easily get that $\Sigma(t, \cdot) : \spL_w \to P({\mathcal L_2^0}_w)$ is u.s.c. for a.e. $t \in [0,k]$. Thus, $\Sigma(t, \cdot) : \spL \to P({\mathcal L_2^0}_w)$ is u.s.c. for a.e. $t \in [0,k]$.

\noindent Now, by Mazur's Theorem, for every $t \in [0,k]$ we obtain a sequence
\begin{equation*}
\widetilde{u}^m_t = \sum_{i=0}^{k_m} \lambda_{m,i}u^{m+i}_t, \quad \lambda_{m,i} \ge 0, \quad \sum_{i=0}^{k_m} \lambda_{m,i} =1
\end{equation*}
\noindent such that $\widetilde{u}^m_t \rightarrow u_t$ in $L^2(\Omega,H)$ and, up to a subsequence, $\widetilde{u}^m_t(\omega) \rightarrow u_t(\omega)$ for a.e. $\omega \in \Omega$. Considering the sequence 
\begin{equation}
  \widetilde{\sigma}^m:= \sum_{i=0}^{k_m} \lambda_{m,i} \sigma^{m+i},
\end{equation}
by the convexity of $\Sigma(t,\cdot)$ we have that
\begin{equation}
  \widetilde{\sigma}^m_t:= \sum_{i=0}^{k_m} \lambda_{m,i} \sigma^{m+i}_t \in \sum_{i=0}^{k_m} \lambda_{m,i} \Sigma(t,u^{m+i}_t) \subset \Sigma(t,\widetilde{u}^m_t).
\end{equation}
Moreover, $\widetilde{\sigma}^m \in L^2([0,k] \times \Omega,\mathcal L_2^0)$ and $\widetilde{\sigma}^m \rightharpoonup \sigma$ in $L^2([0,k] \times \Omega,\mathcal L_2^0)$. Again, by Mazur's Theorem, for every $t \in [0,k]$ we obtain a sequence
\begin{equation*}
\widehat{\sigma}^m = \sum_{j=0}^{j_m} \beta_{m,j}\widetilde{\sigma}^{m+j}, \quad \beta_{m,j} \ge 0, \quad \sum_{j=0}^{j_m} \beta_{m,j} =1,
\end{equation*}
such that $\widehat{\sigma}^m \to \sigma$ in $L^2([0,k] \times \Omega,\mathcal L_2^0)$ and so, up to subsequences, $\widehat{\sigma}^m_{t}(\omega) \to \sigma_t(\omega)$ in $\mathcal L_2^0$ for a.e. $(t,\omega) \in [0,k] \times \Omega$. Let $N_0 \subset [0,k]\times \Omega$ a set with Lebesgue measure zero, such that $\widetilde{u}^m_t(\omega) \to u_t(\omega)$, $\Sigma(t, \cdot) : \spL \to P({\mathcal L_2^0}_w)$ is u.s.c., $\widetilde{\sigma}^m(t) \in \Sigma(t, \widetilde{u}^m(t))$ and $ \widehat{\sigma}^m_t(\omega) \rightarrow \sigma_t(\omega)$ in $\mathcal L_2^0$ for all $(t,\omega) \in [0,k] \times \Omega \setminus N_0$ and $m \in \mathbb{N}$. Fix $(t_0,\omega_0) \notin N_0$ and assume, by contradiction, that $\sigma_{t_0} \notin \Sigma(t_0, u_{t_0})$. Since $\Sigma(t_0, u_{t_0})$ is closed and convex, from the Hahn-Banach Theorem there is a weakly open convex set $V \supset \Sigma(t_0, u_{t_0})$ satisfying $\sigma_{t_0} \notin \overline{V}$. Since $\Sigma(t_0, \cdot) : \spL \to P({\mathcal L_2^0}_w)$ is u.s.c, we can find a neighbourhood $U$ of $u_{t_0}(\omega_0)$ such that $\Sigma(t_0, u) \subset V$ for all $u \in U$.
The convergence $\widetilde{u}^m_{t_0}(\omega_0) \rightarrow u_{t_0}(\omega_0)$ as $m \rightarrow \infty$ implies the existence of $m_0 \in \mathbb{N}$ such that $\widetilde{u}^m_{t_0}(\omega_0) \in U$ for all $m > m_0$. Therefore, $\widetilde{\sigma}^m_{t_0} \in \Sigma(t_0, \widetilde{u}^m_{t_0}) \subset V$ for all $ m > m_0$. Since $V$ is convex, we also have that $\widehat{\sigma}^m_{t_0} \in V$ for all $m > m_0$ and so we get the contradiction $\sigma_{t_0} \in \overline{V}$. Thus, we conclude $\sigma_t \in \Sigma(t, u_t)$ for a.e. $t \in [0,k]$.

\noindent {\bf Step 2.} {\it The multioperator  $\mathscr{F}_k$ is weakly compact.}

\noindent We first prove that $\mathscr{F}_k$ is relatively weakly sequentially compact. To this aim, let $\{u_m\} \subset \mathscr{Q}$ and $\{v_m\} \subset \mathcal{C}_k$ satisfying $ v_m \in \mathscr{F}_k(u_m)$ for all $m \in \mathbb{N}$. By the definition of the multioperator $\mathscr{F}_k$, there exists a sequence $\{\sigma_m \}$, $\sigma_m \in Sel_{\Sigma}^k(u_m)$, such that for all $t \in [0,k]$,
\begin{equation*}
v_m(t) = T(t)u_0 + \int_0^t T(t-s)f(s, u_m(s)) \, ds + \int_0^t T(t-s)\sigma_m(s) \, dW_s.
\end{equation*}

\noindent Further, reasoning as in  Step 1, we have that there exists a subsequence, denoted as the sequence, and a function $\sigma$ such that $\sigma_m \rightharpoonup \sigma$ in $L^2([0,k] \times \Omega,\mathcal L_2^0)$. Since $u_m \in \mathscr{Q}$, according to $(H_3)$, 
\begin{equation}
  \mathbb{E}\norma {f(t,u_m(t))}^2_{H} \le C_f(t)(1+R e^{Lk})  
\end{equation}
for a.e. $t \in [0,k]$ and every $m$, i.e., $\{f(\cdot,u_m(\cdot))\}$ is uniformly integrable in $L^1([0,k],L^2(\Omega,H))$. Hence, by the reflexivity of the space $L^2(\Omega,H)$ and Theorem \ref{weakcompL1}, we have the existence of a subsequence, denoted as the sequence, and a function $\overline{f}$ such that $f(\cdot,u_m(\cdot)) \rightharpoonup \overline{f}$ in $L^2([0,k],L^2(\Omega,H))$. Again, reasoning as in Step 1 we have that
\begin{equation}
v_m(t) \rightharpoonup \overline{v}(t) = T(t)u_0 + \int_0^t T(t-s)\overline{f}(s) \, ds + \int_0^t T(t-s)\sigma(s) \, dW_s, \quad \forall t \in [0,k].
\end{equation}

\noindent Furthermore, by $(H_A), (H_3)$ and $(H_6)$, we have
\begin{align*}
&\mathbb{E}\norma{v_m(t)}_H^2\le 3\mathbb{E}\norma{T(t)u_0}_H^2 +3\mathbb{E} \norma*{\int_0^t T(t-s)f(s, u_m(s)) \, ds}_H^2 \\
& \quad  +3\mathbb{E} \norma*{\int_0^t T(t-s)\sigma_m(s) \, dW_s}_H^2 \\
&\le 3M_1^2\mathbb{E}\norma{u_0}_H^2 +3M_1^2\int_0^t \mathbb{E}\norma*{f(s,u_m(s))}_H^2 \, ds +3M_1^2{\rm tr}(Q)\int_0^t \mathbb{E}\norma{\sigma_m(s)}^2_{\mathcal L_2^0} \,ds \\
&\le 3 M_1^2\mathbb{E}\norma{u_0}_H^2 +3 M_1^2 k (\norma{C_f}_{L^1([0,k], \mathbb{R})} +{\rm tr}(Q)\norma{C_\Sigma}_{L^1([0,k], \mathbb{R})}) (1+e^{Lk}) \le N,
\end{align*}
for some constant $N > 0$. Recalling the definition of weak convergence in $C([0,k], L^2(\Omega, H))$, we have that $ v_m \rightharpoonup \overline{v}$ in $\mathcal{C}^k$. Thus, $\mathscr{F}_k(\mathscr{Q})$ is relatively weakly sequentially compact, hence relatively weakly compact by Theorem \ref{zhou th 1.26}.

\noindent {\bf Step 3.} {\it The multioperator  $\mathscr{F}_k$ has weakly compact and convex values.}

\noindent Fix $u \in \mathscr{Q}$, since $\Sigma$ is convex valued, from the linearity of the integral and the operator $\{T(t)\}_{t \ge 0}$, it follows that the set $\mathscr{F}_k(u)$ is convex. The weak compactness of $\mathscr{F}_k(u)$ follows by Steps 1 and 2.

\noindent {\bf Step 4.} {\it The multioperator $\mathscr{F}_k$ maps the ball $\mathscr{Q}$ into itself.}

\noindent Let $u \in \mathscr{Q}$, and $z \in \mathscr{F}_k(u)$. Then there exists a selection $\sigma \in \mbox{Sel}_\Sigma^k$ such that
\begin{equation*}
z(t) = T(t)u_0 + \int_0^t T(t-s)f(s, u(s)) \, ds + \int_0^t T(t-s)\sigma(s) \, dW_s, \, \forall t\in [0,k].
\end{equation*}
By $(H_A), (H_3)$ and $(H_6)$, we have
\begin{align*}
& \mathbb{E}\norma{z(t)}_H^2 \le 3\mathbb{E}\norma{T(t)u_0}_H^2 +3\mathbb{E} \norma*{\int_0^t T(t-s)f(s, u(s)) \, ds}_H^2 \\
& \quad +3\mathbb{E} \norma*{\int_0^t T(t-s)\sigma(s) \, dW_s}_H^2 \\
&\le 3M_1^2\mathbb{E}\norma{u_0}_H^2 +3M_1^2\int_0^t \mathbb{E}\norma*{f(s,u(s))}_H^2 \, ds +3M_1^2{\rm tr}(Q)\int_0^t \mathbb{E}\norma{\sigma(s)}^2_{\mathcal L_2^0} \,ds \\
&\le 3 M_1^2\mathbb{E} \left(\norma{u_0}_H^2 + \|C_f\|_{L^1([0,k], \mathbb{R})}+{\rm tr}(Q)\norma{C_\Sigma}_{L^1([0,k], \mathbb{R})}\right) \\
& \quad + 3 M_1^2 \int_0^t (C_f(s) + {\rm tr}(Q)C_\Sigma(s)) R e^{Ls} \,ds,
\end{align*}
for all $t \in [0,k]$, where $M_1$ is defined in \eqref{M1}. From the choices of $L$ and $R$ in \eqref{RL} we have that
\begin{align*}
 \mathbb{E}\norma{z(t)}_H^2 \leq C_1 + R e^{Lt} q \leq Re^{Lt}, 
\end{align*}
where $C_1$ is defined in \eqref{C1}. Thus $z \in \mathscr{Q}$. 

\noindent Since $\mathscr{F}_k^n$ is a weakly compact operator, the set $V = \overline{\mathscr{F}_k(\mathscr{Q})}^w$ is weakly compact. Let now $\widetilde{V} = \overline{co}(V)$, where $\overline{co}(V)$ denotes the closed convex hull of $V$. By Theorem \ref{zhou th 1.27}, $\widetilde{V}$ is a weakly compact set. Moreover, from the fact that $\mathscr{F}_k(\mathscr{Q}) \subset \mathscr{Q}$ and $\mathscr{Q}$ is a convex closed set we have that $\widetilde{V} \subset \mathscr{Q}$ and hence
\begin{equation*}
\mathscr{F}_k(\widetilde{V})= \mathscr{F}_k(\overline{co}(\mathscr{F}_k(\mathscr{Q}))) \subseteq \mathscr{F}_k(\mathscr{Q}) \subseteq \overline{\mathscr{F}_k(\mathscr{Q})}^w = V \subset \widetilde{V}.
\end{equation*}
Hence there exists a fixed point $u \in \mathscr{F}_k(u)$, thus a mild solution of \eqref{inclusione-approssimata}.

\noindent {\bf Step 5.} {\it The set of solutions to problem \eqref{inclusione-approssimata} is weakly compact in $\mathcal{C}^k$}.

\noindent Being $\mathscr{F}_k$ weakly closed, the fixed point set $\mbox{Fix}(\mathscr{F}_k)$ is weakly closed. Moreover, $\mbox{Fix}(\mathscr{F}_k)\subset \mathscr{Q}$ hence it is a bounded set. Finally, since $\mathscr{F}_k$ is weakly compact and, by definition, 
$$
\mbox{Fix}(\mathscr{F}_k)\subset\mathscr{F}_k(\mbox{Fix}(\mathscr{F}_k)),
$$
it follows that $\mbox{Fix}(\mathscr{F}_k)$ is a weakly compact set.
\end{proof}

\section{Proof of Theorem \ref{th-main}} \label{sec:proof}

\begin{proof}
According to Theorem \ref{zhou th 8.1}, for every $k \in \mathbb{N}_+$ there exists a solution $u_k \in \mathcal{C}^k$ to the problem \eqref{inclusione-approssimata} with $ \|u_k(t)\| \leq n $ and $ t \in [0,k].$ Define
$$
\widetilde{u}_{k}(t)=
\begin{cases}
u_{k}(t)\ \ \text{for}\ \ t\in [0,k],\\
u_{k}(k)\ \ \text{for}\ \ t\geq k.
\end{cases}
$$
Consider the restriction of the sequence of mild solutions $\{\widetilde{u}_k\}$ to the interval $[0,1]$, namely $\{\widetilde{u}_k^1\}:=\{\widetilde{u}_k \big{|}_{[0,1]}\}$. Since, again by Theorem \ref{zhou th 8.1}, the set of mild solutions of problem \eqref{inclusione-approssimata} is weakly compact in $C([0,1],L^2(\Omega,H))$, there exists a subsequence, still denoted as the sequence, weakly converging to a function $\psi^1\in C([0,1],L^2(\Omega,H))$ mild solution of the problem \eqref{inclusione-approssimata} in $[0,1]$. Moreover, since $\mathcal{C}$ is a closed subspace of $C([0,+\infty),L^2(\Omega,H))$ it follows that, for every $k \geq 1$, $\mathcal{C}^k$ is a closed subspace of $C([0,k],L^2(\Omega,H))$ and so $\psi^1 \in \mathcal{C}^1$. Now, let us consider the sequence $\{\widetilde{u}_k^2\}_{k \ge 2}:=\{\widetilde{u}_k \big{|}_{[0,2]}\}_{k \ge 2}$. Again, we get that there exists a subsequence, still denoted as the sequence, weakly converging to a function $\psi^2 \in \mathcal{C}^2$ mild solution for the problem \eqref{inclusione-approssimata} in $[0,2]$. By construction it follows that
\begin{equation*}
\psi^2 \big{|}_{[0,1]}=\psi^1\ .
\end{equation*}
By iterating this process we obtain, for every $k\in \mathbb{N}_+$,
a mild solution $\psi^k:[0,k]\to L^2(\Omega,H)$ for problem \eqref{inclusione-approssimata} such that for every integer $k\ge 2$, we have
\begin{equation}
\label{past} \psi^k \big{|}_{[0,k-1]}=\psi^{k-1} \ .
\end{equation}
Hence by the continuity of the maps $\psi^k$ for any $k \in \mathbb{N}_+$ and by \eqref{past} we have that the map $\psi:[0,+\infty) \to L^2(\Omega,H)$ defined as
$$
\psi(t)=\left \{ \begin{array}{ll}
                     \psi^1(t) & t \in [0,1] \\
                     \psi^2(t) & t \in ]1,2] \\
                     \dots & \dots \\
                     \psi^k(t) & t \in ]k-1,k] \\
                     \dots & \dots
                     \end{array}
                     \right.
$$
is a mild solution to the problem \eqref{inclusione} on $[0,+\infty)$ with $\psi \in \mathcal{C}$. Indeed, by definition of mild solution (Definition \ref{mild}), for every $k \in \mathbb{N}_+$ there exists a map $\sigma^k \in Sel_{\Sigma}^k(\psi^k)$. Thus, the map $\sigma \in L^2([0,+\infty),\mathcal L_2^0)$ defined as
$$
\sigma(t)=\left \{ \begin{array}{ll}
                     \sigma^1(t) & t \in [0,1] \\
                     \sigma^2(t) & t \in ]1,2] \\
                     \dots & \dots \\
                     \sigma^k(t) & t \in ]k-1,k] \\
                     \dots & \dots
                     \end{array}
                     \right.
$$
is a selection of $\Sigma$ corresponding to $\psi$. In fact, for every $t \in [0,+\infty)$ there exists $k \in \mathbb{N}_+$ such that $t \in [k-1,k]$ and so $\sigma(t)=\sigma^k(t) \in \Sigma(t,\psi^k(t))=\Sigma(t,\psi(t))$. Moreover, by construction the maps $\psi$ and $\sigma$ satisfy the equation \eqref{eqmild}.
\end{proof}

\section{Existence of nonnegative solutions} \label{sec:nonneg}

\noindent In this section we study the sign of the solution of the problem \eqref{inclusione} under the hypotheses $(H_A)$, $(H_1)-(H_6)$ and requiring additionally that $A$ generates a positive contraction semigroup and the following assumption:
\begin{itemize} 
    \item[($H_7$)] for every $t \in [0,+\infty)$ and $h \in \spL$,
    \begin{equation*}
        - \braket{h^-, f(t, h)} + \norma{ \mathbbm{1}_{\{h \le 0\}}\Sigma(t,h)}^2_{\mathcal L_2^0} \le \norma{h^-}^2_H
    \end{equation*}
    with
\begin{equation*}
    h^- = 
    \begin{cases}
    -h & h < 0 \\
     0 & h \ge 0,
    \end{cases}
\end{equation*}
\end{itemize}
where $\mathbbm{1}_A$ is the characteristic function of a set $A$ and the ordering in $H$ is induced by a cone $\mathcal{K} \subset H$, i.e. $h \ge 0$ if and only if $h \in \mathcal{K}$ and $h < 0$ means $h \in \mathcal{K}^c$, where $\mathcal{K}^c$ denotes the complementary set of $\mathcal{K}$. 

\begin{rem}
To prove the existence of a nonnegative mild solution we use the following notation for $t\in [0,+\infty)$ and $w\in \Omega$ 
\begin{equation}
\label{u-}
      u_t^-(w)=
       \begin{cases}
       -u_t(w) \quad & \text{if} \,\, u_t(w) < 0 \\
       0 \quad \quad & \text{if} \,\, u_t(w) \ge 0
       \end{cases}
\end{equation}
   where since $u_t$ is a function with values in $L^2(\Omega, H)$, for each $t \in [0,+\infty)$ and $w \in \Omega$ an equivalent class is specified. So, if we write $u_t(w) \ge 0$, or $u_t \ge 0$ we mean that in the corresponding class there is a function belonging to $\mathcal{K}$.
\end{rem}

\noindent We state here the main result of the section.
\begin{thm}
\label{th marinelli}
Under the assumptions $(H_A)$, $(H_1) - (H_7)$, if in addition $A$ is the generator of a positive contraction semigroup and $u_0 \in L^2(\Omega,H)$ is nonnegative, there exists at least one nonnegative mild solution $u \in \mathcal{C}$ of \eqref{inclusione}.
\end{thm}

\begin{proof}

\noindent Let $\overline{u}$ be a mild solution of \eqref{inclusione} which exists in view of Theorem \ref{th-main}, namely there exists a map $\overline{\sigma} \in L^2([0,+\infty),\mathcal{L}^0_2)$, $\overline{\sigma}_t \in \Sigma(t,\overline{u}_t)$ for a.e. $t \in [0,+\infty)$, such that
\begin{equation*}
\overline{u}_t = T(t)u_0 + \int_0^t T(t-s)f(s, \overline{u}_s)\,ds + \int_0^t T(t-s)\overline{\sigma}_s\,dW_s, \,\, \text{for} \,\, t \in [0,+\infty).
\end{equation*}
Recalling that, for $\lambda > 0$, $J^\lambda \in \mathcal{L}(H, D(A))$ (see (iii) in Proposition \ref{propertiesJlambda}) we consider the following approximations
    \begin{align}
    u^\lambda_0 & =\lambda J^\lambda u_0 \in \mathcal{M}(\Omega,D(A)) \\
    f^\lambda_s & =\lambda J^\lambda f(s,\overline{u}_s) \in \mathcal{M}(\Omega, L^2([0,+\infty),D(A))) \\
    \sigma^\lambda_s & =\lambda J^\lambda \overline{\sigma}_s \in \lambda J^\lambda \Sigma(s,\overline{u}_s) \in \mathcal{M}(\Omega, L^2([0,+\infty),\mathcal{L}_2(Q^{\frac{1}{2}}K,D(A)))
    \end{align}
    where $\mathcal{M}(\Omega,E)$ denotes the space of measurable maps from $\Omega$ to a Banach space $E$ and $u^\lambda_0$ is $\mathcal{F}_0$-measurable. Notice that for $z \in K$
    \begin{equation}
    \sigma^\lambda_s(z) = \sum_{k=1}^{+ \infty} \lambda J^\lambda b_k \braket{z, e_k}_K
    \end{equation}
    where $\{e_k\}$ is an orthonormal basis in $K$ and $\{b_k\}$ is a sequence of elements in $H$. Denote by $\{u^\lambda_t\}$ the mild solution of the problem 
\begin{equation}
\label{inclusionelambda}
\begin{cases}
du_t = [Au_t + f^\lambda_t]dt + \sigma^\lambda_t dW_t, \quad t > 0 \\
u(0)=u^\lambda_0,
\end{cases}
\end{equation}
see e.g. \cite[Theorem 7.2]{DaPrato}. Namely, $\{u^\lambda_t\}$ is defined as
    \begin{equation*}
u^\lambda_t = T(t)u^\lambda_0 + \int_0^t T(t-s)f^\lambda_s\,ds + \int_0^t T(t-s)\sigma^\lambda_s\,dW_s, \,\, \text{for} \,\, t \in [0,+\infty).
\end{equation*}
Notice that, since $u^\lambda_t \in D(A)$ for every $t \in [0,+\infty)$, it follows that it is also a strong solution of \eqref{inclusionelambda}, see Remark \ref{mild-strong}.

\noindent Now, we apply the It\^o  formula (see Proposition \ref{itoformula}) to the process $g((u^\lambda_t)^-)$, where the map $g:H \to \mathbb{R}$ is defined as $g(x)=\frac{1}{2}\|x\|^2_H,\; x \in H$ and for $t \in [0,+\infty)$, $(u^\lambda_t)^-$ is defined in \eqref{u-}, obtaining
\begin{equation}
\label{ito-ulambda}
\begin{aligned}
    \frac{1}{2}\norma{(u^\lambda_t)^-}^2_\spL &= \frac{1}{2}\norma{(u^\lambda_0)^-}^2_\spL - \int_0^t \braket{A u^\lambda_s, (u^\lambda_s)^-} \, ds - \int_0^t \braket{f^\lambda_s, (u^\lambda_s)^-} \, ds \\ 
    & \quad \quad + \frac{1}{2}\int_0^t \|\sigma^\lambda_s \mathbbm{1}_{\{u^\lambda \le 0\}}\|^2_{\mathcal L_2^0}  \, ds - \int_0^t (u^\lambda_s)^- \sigma^\lambda_s \, dW_s.
\end{aligned}
\end{equation}
Recalling that $\braket{Au,u} \leq 0$ for every $u \in D(A)$ (see (vi) in Proposition \ref{propertiesJlambda}) and the definition of positive semigroup (see Definition \ref{positive semigroup}) we have that
\begin{align}
    \braket{Au,u^-}&=\braket{Au^+,u^-}-\braket{Au^-,u^-} \geq \braket{\displaystyle\lim_{t\to 0^+} \frac{T(t)u^+-u^+}{t},u^-} \\
    &=\displaystyle\lim_{t\to 0^+} \frac{1}{t}\left(\braket{T(t)u^+,u^-}-\braket{u^+,u^-}\right) \geq 0, \quad \mbox{for every} \quad u \in D(A).
\end{align}
Hence, the second term on the left-hand side in \eqref{ito-ulambda} is negative. Moreover, thanks to assumption $(H_7)$, the sum of the third and fourth terms on the right-hand side of \eqref{ito-ulambda} can be estimated by 
\begin{equation*}
    \int_0^t \norma{(u^\lambda_s)^-}^2_\spL \, ds, \quad t \in [0,+\infty).
\end{equation*}
We are thus left with 
\begin{align}
\label{eq 6 Marinelli}
    \norma{(u^\lambda_t)^-}^2_\spL \le & \norma{(u^\lambda_0)^-}^2_\spL + 2 \int_0^t \norma{(u^\lambda_s)^-}^2_\spL \, ds \\
    & - 2 \int_0^t (u^\lambda_s)^- \sigma^\lambda_s \, dW_s, \quad t \in [0,+\infty).
\end{align}

\noindent Let $(T_n)$ be a localizing sequence for the continuous local martingale on the right-hand side of\eqref{eq 6 Marinelli}.
\newline
Introducing the stopped process $u^\lambda(n) := u^{\lambda,T_n}$, one has
\begin{align*}
    \norma{(u^\lambda_t(n))^-}^2_\spL &\le \norma{(u^\lambda_0)^-}^2_\spL + 2 \int_0^{t \wedge T_n} \norma{(u^\lambda_s)^-}^2_\spL \, ds - 2 \int_0^{t \wedge T_n} (u^\lambda_s)^- \, \sigma^\lambda_s \, dW_s \\
    &\le \norma{(u^\lambda_0)^-}^2_\spL + 2 \int_0^{t} \norma{(u^\lambda_s(n))^-}^2_\spL \, ds - 2 \int_0^{t \wedge T_n} (u^\lambda_s)^- \, \sigma^\lambda_s \, dW_s,
\end{align*}
for every $t \in [0,+\infty)$.
Recalling that $\mathbb{E} \norma{u_0}^2_H < \infty$ by assumption, taking expectation on both side, applying Tonelli's theorem  and the zero mean property of the stochastic integral, we get
\begin{equation*}
    \mathbb{E}\norma{(u^\lambda_t(n))^-}^2_H \le \mathbb{E}\norma{(u^\lambda_0)^-}^2_H + 2 \int_0^{t} \mathbb{E}\norma{(u^\lambda_s(n))^-}^2_H \, ds
\end{equation*}
hence also, by Gronwall's inequality,
\begin{equation*}
    \mathbb{E}\norma{(u^\lambda_t(n))^-}^2_H \le \widetilde{C} \, \mathbb{E}\norma{(u^\lambda_0)^-}^2_H \quad \forall \; t \in [0,+\infty),
\end{equation*}
for a suitable constant $\widetilde{C} > 0$. Passing to the limit as $n \to \infty$, Fatou's lemma yields 
\begin{equation}
\label{disuguaglianza con u0}
    \mathbb{E}\norma{(u^\lambda_t)^-}^2_H \le \widetilde{C} \, \mathbb{E}\norma{(u^\lambda_0)^-}^2_H \quad \forall \,t \in \; [0,+\infty).
\end{equation}
Now, fix $t \in [0,+\infty)$, recalling that $\lambda J^\lambda \to I$ in $\mathcal{L}(H,H)$ with respect to the strong operator topology (see (v) in Proposition \ref{propertiesJlambda}) we have that
 \begin{align}
    & u^\lambda_0 \to u_0 \quad \mathbf{P}-\mbox{a.s. in} \; H \\
    & f^\lambda_t \to f(t,\overline{u}_t) \quad \mathbf{P}-\mbox{a.s. in} \; H \\
    & \sigma^\lambda_t \to \overline{\sigma}_t \quad \mathbf{P}-\mbox{a.s. in} \; \mathcal L_2^0,
    \end{align}
where the last convergence is due to Theorem \ref{convergenceHilbertSchmidt}. Moreover, for every $t \in [0,+\infty)$, exploiting the It\^o isometry (see Proposition \ref{itoisometry}), we have
\begin{equation}
\label{ulambdaconvergence}
\begin{split}
    \mathbb{E}\norma{u^\lambda_t-\overline{u}_t}^2_H \leq & 3 M^2 \norma{u^\lambda_0-u_0}^2_H + 3 M^2 \displaystyle\int_0^t \mathbb{E}\norma{f^\lambda_s-f(s,\overline{u}_s)}^2_H \,ds \\
    & + 3 M^2 {\rm tr}(Q)\displaystyle\int_0^t \mathbb{E}\norma{\sigma^\lambda_s-\overline{\sigma}(s)}^2_{\mathcal L_2^0} \,ds.
    \end{split}
\end{equation}
Recalling that, $\|\lambda J^\lambda\|_{\mathcal{L}(H,H)} \leq 1$ for every $\lambda > 0$, it follows
 \begin{align}
    & \|u^\lambda_0(\omega)\|^2_H \leq \|u_0(\omega)\|^2_H \quad \mbox{for a.e.} \quad \omega \in \Omega \\
    & \|f^\lambda_t(\omega)\|^2_H \leq \|f(t,\overline{u}_t(\omega))\|^2_H \quad \mbox{for a.e.} \quad t \in [0,+\infty) \quad \mbox{and} \quad \omega \in \Omega \\
    & \|\sigma^\lambda_t(\omega)\|^2_{\mathcal L_2^0} \leq \|\overline{\sigma}_t(\omega)\|^2_{\mathcal L_2^0} \quad \mbox{for a.e.} \quad t \in [0,+\infty) \quad \mbox{and} \quad \omega \in \Omega.
    \end{align}
Now, by assumption $(H_3)$, and since by Proposition \ref{boundsol} there exists $\overline{C} > 0$ such that $\norma{\overline{u}}_\mathcal{C} \leq \overline{C}$, we have that
\begin{align}
    & \mathbb{E}\|f(t,\overline{u}_t)\|^2_H \leq C_f(t)(1+\mathbb{E}\|\overline{u}_t\|^2_H) \leq C_f(t)(1+\overline{C}), \quad \mbox{for a.e.} \; t \in [0,+\infty),
\end{align}
getting that $f(t,\overline{u}_t) \in L^2(\Omega,\mathcal L^0_2)$ for a.e. $t \in [0,+\infty)$. Analogously by $(H_6)$
\begin{align}
&\mathbb{E}\|\overline{\sigma}_t\|^2_{\mathcal L_2^0} \leq C_\Sigma(t)(1+\overline{C}) \quad \mbox{for a.e.} \; t \in [0,+\infty),
    \end{align}
    obtaining that $\overline{\sigma}_t \in L^2(\Omega,\mathcal L^0_2)$ for a.e. $t \in [0,+\infty)$ too. Taking into account the additional assumption $u_0 \in L^2(\Omega,H)$, by the Lebesgue Dominated Convergence Theorem, we get for a.e. $t \in [0,+\infty)$ that
\begin{align}
    & u^\lambda_0 \to u_0 \quad \mbox{in} \; L^2(\Omega,H) \\
    & f^\lambda_t \to f(t,\overline{u}_t) \quad \mbox{in} \; L^2(\Omega,H) \\
    & \sigma^\lambda_t \to \overline{\sigma}_t \quad \mbox{in} \; L^2(\Omega,\mathcal L_2^0).
    \end{align}    
Furthermore, recalling that $C_f$ and $C_\Sigma$ belong to $L^1([0,+\infty),\mathbb{R}_+)$ we can apply the Lebesgue Dominated Convergence Theorem also in $L^1([0,+\infty),L^2(\Omega,H))$ and in $L^1([0,+\infty),L^2(\Omega,\mathcal L^0_2))$ yielding
\begin{align}
    & f^\lambda \to f(\cdot,\overline{u}_\cdot) \quad \mbox{in} \; L^1([0,+\infty),L^2(\Omega,H)) \\
    & \sigma^\lambda \to \overline{\sigma} \quad \mbox{in} \; L^1([0,+\infty),L^2(\Omega,\mathcal L_2^0).
    \end{align} 
Finally, by \eqref{ulambdaconvergence} we obtain that $u^\lambda_t$ converges to $\overline{u}_t$ in $L^2(\Omega,H)$ for a.e $t \in [0,+\infty)$. Thus, by \eqref{disuguaglianza con u0}, we have
 \begin{equation}
    \mathbb{E}\norma{\overline{u}_t^-}^2_H \le \mathbb{E}\norma{(u^\lambda_t)^-}^2_H \le \widetilde{C} \, \mathbb{E}\norma{(u^\lambda_0)^-}^2_H \leq \widetilde{C} \, \mathbb{E}\norma{u_0^-}^2_H = 0, \quad \forall t \in [0,+\infty).
\end{equation}
From which it follows that the nonnegativity of $u_0$ implies the nonnegativity of $u_t$ for all $t \in [0,+\infty)$.
\end{proof}

\noindent If we assume the existence of at least one solution to \eqref{inclusione}, the nonnegativity of such solutions can still be obtained without requiring the condition $\mathbb{E}\norma{u_0}^2_H < \infty$. More precisely, we can establish the following result.
\begin{thm}
\label{th marinelli2}
Under the assumptions $(H_A)$, $(H_1) - (H_7)$, if in addition $A$ is the generator of a positive contraction semigroup and $u_0$ is nonnegative every mild solution $u \in \mathcal{C}$ of \eqref{inclusione} is nonnegative .
\end{thm}
\begin{proof} The proof is divided into two steps.

\noindent {\bf Step 1.} We assume that $\mathbb{E}\norma{u_0}^2_H < \infty$. Following the proof of Theorem \ref{th marinelli} we get that the solution $u$ of \eqref{inclusione} is nonnegative.

\noindent {\bf Step 2.} Consider a mild solution $v$ of the problem \eqref{inclusione}, the sequence of random variables $\{u_0^m\}$, with $u_0^m=\mathbbm{1}_{\{u_0 \le m\}} u_0$, $m \in \mathbb{N}_+$ and the following family of stochastic differential inclusions with $u_0^m$ as initial condition:
\begin{equation}
\label{inclusioneu0m}
\begin{cases}
du_t \in [Au_t + f(t, v_t)]dt + \Sigma(t,v_t)dW_t, \quad t > 0 \\
u(0)=u_0^m.
\end{cases}
\end{equation} 
For every $m \in \mathbb{N}_+$, by the boundedness of $\{u_0^m\}$ and Step 1,  there exists a nonnegative mild solution $u^m$ of problem \eqref{inclusioneu0m} given, for every $t \in [0,+\infty)$, by
\begin{equation*}
u^m_t = T(t)u^m_0 + \int_0^t T(t-s)f(s, v_s) \, ds + \int_0^t T(t-s)\sigma_s \, dW_s, 
\end{equation*}
where $\sigma \in Sel_\Sigma (v)$. 
Moreover, trivially $u_0^m \to u_0$ in $H$, hence for every $t \in [0,+\infty)$
\begin{equation}
    u^m_t \to T(t)u_0 + \int_0^t T(t-s)f(s, v_s) \, ds + \int_0^t T(t-s)v_s \, dW_s = v_t.
\end{equation}
In conclusion, by the lower semicontinuity of the norm we have
\begin{equation}
    \mathbb{E}\|v_t^-\|^2_H \leq \mathbb{E}\|{u^m_t}^-\|^2_H = 0,
\end{equation}
obtaining the nonnegativity of $v$.
\end{proof}

\subsection*{Acknowledgments}

The second-named and third-named  authors are members of Gruppo Nazionale per l'Analisi Matematica, la Probabilità e le loro Applicazioni (GNAMPA) of Istituto Nazionale di Alta Matematica (INdAM). 
The second-named author was partially supported by  the European Union-Next Generation EU - PRIN research project n. 2022ZXZTN2.
The third-named author was partially supported by  the European Union-Next Generation EU - PRIN research project n. 2022FPLY97.

\bibliographystyle{plainnat}
\bibliography{biblio_ABC.bib}

\end{document}